\documentclass{amsart}
\usepackage{amsmath}
\usepackage{amsthm}
\usepackage{amsfonts,subfigure}
\usepackage{tikz}
\usetikzlibrary{arrows, calc}
\usetikzlibrary{patterns, shapes, decorations.markings}
\newcommand{\ccdot}{\;\cdot\;}

\newcommand{\simI}{\overset{1}{\sim}}
\newcommand{\simII}{\overset{2}{\sim}}
\newcommand{\regionP}{\Gamma}
\newcommand{\regionO}{\Lambda}
\newcommand{\E}{\mathbb{E}}
\newcommand{\T}{T}
\newcommand{\A}{A}

\newcommand{\RR}{\mathbb{R}}
\newcommand{\NN}{\mathbb{N}}
\newcommand{\PP}{\mathbb{P}}

\newcommand{\eqdef}{=}
\newcommand{\ip}[1]{\langle #1 \rangle}

\newcommand{\sgn}{\textrm{sgn}}
\newcommand{\ep}{\epsilon}

\newtheorem{theorem}{Theorem}
\newtheorem{lemma}{Lemma}[section]

\newtheorem{proposition}[lemma]{Proposition}

\newtheorem{corollary}[lemma]{Corollary}
\newtheorem{remark}[lemma]{Remark}

\newtheorem{definition}{Definition}

\title[Propagating Lyapunov Functions]{Propagating Lyapunov Functions to Prove Noise--induced Stabilization}
\author[Avanti Athreya]{Avanti Athreya$^{1,2}$}
\author[Tiffany Kolba]{Tiffany Kolba$^{1,3}$}
\author[Jonathan C. Mattingly]{Jonathan C. Mattingly$^1$}

\begin{document}
\maketitle

\footnotetext[1]{Department of Mathematics, Duke University, Durham,
  NC 27708: (jonm@math.duke.edu)}
\footnotetext[2]{Current Address: Department of Applied Mathematics and Statistics,
  The Johns Hopkins University, Baltimore, MD 21218 :
  (dathrey1@jhu.edu)}
\footnotetext[3]{Current Address: Department of Mathematics and Computer Science, Valparaiso University, Valparaiso,
  IN 46385: (tiffany.kolba@valpo.edu)}

\begin{abstract}
  We investigate an example of noise-induced stabilization in the plane that was also considered  in
   (Gawedzki, Herzog, Wehr 2010) and (Birrell, Herzog, Wehr 2011). We show that
  despite the deterministic system not being globally stable, the
  addition of additive noise in the vertical direction leads to a
  unique invariant probability measure to which the system converges at a
  uniform, exponential rate. These facts are established primarily
  through the construction of a
  Lyapunov function which we generate as the solution to a sequence of Poisson equations. Unlike a number of other works, however, our Lyapunov function is constructed in a systematic way, and we present a meta-algorithm we hope
  will be applicable to other problems.
We conclude by proving positivity properties of the
transition density by using Malliavin calculus via some unusually explicit calculations.
\end{abstract}

\section{Introduction}
\label{sec:introduction}

Stabilization by noise is a mathematically intriguing phenomenon. For instance, in
the classic example of the inverted pendulum, the addition of noise
opens up a small neighborhood of local stability around a
deterministically unstable fixed point \cite{ArnoldKleimann1983,Gitterman2005}; in the striking
examples of \cite{scheutzow93}, the addition of noise leads to global
stabilization. In general, however, there are few rigorous proofs of
this phenomenon for specific systems, and most existing proofs depend
upon correctly ``guessing" a Lyapunov function and then verifying that
it satisfies the requisite properties.

In three recent, interesting works
\cite{herzogGawedzkiWehr11,herzogBirrellWehr11,doering11}, a global
Lyapunov function is constructed by patching together functions which
are locally Lyapunov in a collection of regions whose union covers all
of the possible routes to infinity. These papers are concerned with
specific examples in which the stabilization by noise is a property of
the global dynamics rather than the local dynamics near a fixed
point. As such, they are closer in spirit to \cite{scheutzow93} than
to examples such as the inverted pendulum. In these examples, the noise is only important in localized
regions of phase space, but its effect is global, in that it changes
the global nature of the flow. This nature is hinted at in the
patchwork constructions used in
\cite{herzogGawedzkiWehr11,herzogBirrellWehr11,doering11}. In
\cite{doering11}, this structure is the most explicit; there, local
asymptotic expansions are used to construct a patchwork of local
Lyapunov functions. Still in all three works, the local constructions have mainly the
flavor of ``guess--and--check" with some information of the presumed
overall structure of the transport in phase space.

Here we take a more systematic approach to proving global
stabilization by noise; we outline a meta-algorithm which we hope can
be used to produce Lyapunov functions in a number of different
dynamical systems.  Inspired by the examples in
\cite{herzogGawedzkiWehr11,herzogBirrellWehr11, doering11}, we apply
our meta-algorithm to a system of stochastic differential equations in the
plane whose underlying
deterministic dynamics display finite-time blow-up for certain initial
conditions, but in which the addition of an arbitrarily small amount
of noise leads to an invariant probability measure. We consider
essentially the same system as in \cite{herzogGawedzkiWehr11} and
one of the examples in \cite{herzogBirrellWehr11} with the specific
choice of parameters $\alpha_1=\alpha_2=1$ and the change of
coordinates induced by $(x,y)
\mapsto (-x,-y)$. The choice of parameters is representative of the
stable regime we are interested in.   We demonstrate the
existence of an invariant measure by constructing a Lyapunov
function. Our general approach is to build local Lyapunov functions as
solutions to associated partial differential equations (PDEs), where
the PDEs are defined in regions delineated by different asymptotic
behaviors of the flow.  While it is related to that in
\cite{herzogGawedzkiWehr11,herzogBirrellWehr11}, the Lyapunov function
we construct might better be called a ``super'' Lyapunov function, in
that it enables us to prove a stronger form of convergence than in
\cite{herzogGawedzkiWehr11,herzogBirrellWehr11}. Our exponential
convergence results apply equally well to the case of degenerate
stochastic forcing while those in \cite{herzogBirrellWehr11} only
prove exponential convergence in the uniformly elliptic setting. Our
analysis also adapts to the specifics of the problem and is likely to
produce closer-to-optimal results. 

In
\cite{DupuisWilliams1994,HuangKontoyiannisMeyn2001,FortMeynMoulinesPriouret2008,Meyn2008},
the scaling limit of a discrete time Markov chain, called the ``fluid limit''
 in the context of these works, is used to build a
Lyapunov function. In some ways this is related in spirit to our work in this paper. However, we are explicitly interested in the case where the noise  and
its fluctuations are
fundamentally  important to the behavior at infinity. The naive fluid
limit model, however, is a singular limit and hence does not capture the behavior of those systems in which noise plays an essential role. We will see in our example that the naive fluid limit---which
is the underlying unperturbed dynamical system---is in fact unstable. Our
constructions capture the essential stochasticity in the regions where
it matters at infinity.

Combining  our Lyapunov function with a result on the positivity of
transition densities, we prove a strong result on the convergence to
equilibrium of our specific dynamical system. Though the positivity result is
neither the most general nor the most powerful, it is nevertheless of independent
interest, since the proof employs sophisticated ideas from control theory
and Malliavin calculus in a very concrete and transparent way. We hope
that it will help develop the readers intuition in such matters.

\section{Lyapunov Functions}

We are interested in the stability of a Markov process $(X_t, Y_t)$
which is the solution to a stochastic differential equation (SDE) on
the state space $\RR^2$ with generator $L$.  In the deterministic
setting, a Lyapunov function is a positive function of the state space
which decreases, often exponentially, along trajectories. In the
stochastic setting, one requires that the function decrease on
average. More precisely, we define a Lyapunov function $V$ on an
unbounded set $\mathcal{R} \subset \RR^2$ as follows:
\begin{definition}\label{def:Lyap}
  A $C^2$ function $V:\mathcal{R} \rightarrow (0,\infty) $ is a
  \textsc{Lyapunov function} on $\mathcal{R}$ if
\begin{enumerate}
\item $V(x,y) \rightarrow \infty$ as $|(x,y)|\rightarrow \infty$ with  $(x,y)
  \in \mathcal{R}$,
\item there exist constants $m, b>0$ and  $\gamma  > 0$ such that for all $(x,y)\in \mathcal{R}$,
\begin{equation*}
  (L V)(x,y) \leq -m V ^{\gamma}(x,y) + b\,.
\end{equation*}
\end{enumerate}
We say that $V$ is a \textsc{super Lyapunov function} on $\mathcal{R}$
if $\gamma >1$ and a  \textsc{standard Lyapunov function} on
$\mathcal{R}$ if $\gamma=1$. We call $\gamma$ the \textsc{Exponent} of
the Lyapunov function.
If $\mathcal{R}$ is a strict subset of $\RR^2$, we
say that $V$ is a {\em local} (super/standard) Lyapunov function; if
$\mathcal{R}$ is all of $\RR^2$, we say that $V$ is a {\em global}
(super/standard) Lyapunov function.
\end{definition}
We remark that there are several different notions of a Lyapunov
function in the literature, but the one above will be used in this
particular paper. See for example \cite{Hasminskiui1980,MeynTweedie1993,scheutzow93,Veretennikov1997,Veretennikov1999}.

\begin{remark} Notice that the continuity of $V$ coupled with its
  growth at infinity implies that the sub-level sets $\{(x,y) \in \mathcal{R}: V(x,y)\leq R\}$ are
compact for all $R$. In a certain sense, this is the more fundamental
condition, but we will not belabor this point here. See, for example,
Proposition~5.1 in \cite{HairerMattingly2009} for more details.
\end{remark}

It is well-known that the existence of a global Lyapunov function
satisfying the properties in Definition \ref{def:Lyap} implies the
existence of an invariant measure
\cite{Hasminskiui1980,scheutzow93,MeynTweedie1993,HairerMattingly2009}. If
one adds a mild mixing/minorization condition and assumes that the
Lyapunov function is a standard Lyapunov function, it is possible to prove exponential
convergence to this invariant measure\cite{MeynTweedie1993,hairerMattingly08}. If the Lyapunov function is
weaker ($\gamma <1$), then the convergence is generally slower
\cite{MeynTweedie1993,Veretennikov1997,Veretennikov1999,DoucFortGuillin2009,HairerMarkov2010}. in
this case, one might rightly call $V$ a sub Lyapunov function.

For the system considered in this paper, we show the
existence of a global super Lyapunov function, along with needed mild mixing/minorization
conditions. Together, these imply that the rate of convergence to equilibrium is not
only exponential, but also independent of initial
condition.

\subsection{General Construction Strategy}
\label{sec:construction}
We now give an outline of our general approach to constructing a
Lyapunov function. Since many of the details of the implementation
depend on the specific example under consideration, this outline is
meant as an overall rubric. On first reading, this section may seem
rather heuristic and overly vague. We encourage the reader to take it as motivation at
first and then reread this section after
Section~\ref{sec:dominantBalance} and Section~\ref{sec:localLyap}; in
these later sections, the following abstract discussion is made
concrete.

In our general construction algorithm, we begin by identifying a
region in phase space where there is an obvious choice of a Lyapunov
function.  We refer to this region as the ``priming'' region; it is
often characterized as a subset of phase space in which the
deterministic flow is directed toward the origin. We refer to the
associated local Lyapunov function in this region as the ``priming''
Lyapunov function.  Next, by contrast, we identify the region in which
the deterministic dynamics exhibit instability---for example, blow-up
in finite time---and for which noise is essential to the
stabilization, at least insofar as the noise ensures that the
system leaves this region. Since this region is noise-dominated to
some degree, we refer to this region as the ``diffusive" region.
The construction of a local Lyapunov function in this diffusive region
is a key component of our methodology.  We then ``propagate'' the
priming Lyapunov function, from the priming region to the diffusive
region, through a series of intermediate regions of phase space, which
we call ``transport" regions, until we have covered all possible
routes to infinity.  Following this prescription, we  obtain a sequence of local Lyapunov
functions, which we mollify to obtain a global Lyapunov
function.
Beyond this general, overarching strategy is a philosophy of
determining the relevant scaling at infinity in each of the above regions and
systematically producing local Lyapunov function which respect this
scaling.

To determine more precisely the boundaries of these regions, we study
the scaling of the generator $L$ of the SDE as $|(x,y)| \rightarrow
\infty$. Each region corresponds to a different dominant balance of
the terms in the generator.  (See \cite{benderOrszag,white} for a discussion of the
concept of dominate
balances in asymptotic analysis and Section~\ref{sec:dominantBalance}
for the details in our setting.) In order to facilitate the mollification
of the local Lyapunov functions, we choose the regions so that the
intersection between adjoining regions is both nonempty and unbounded.
Neglecting all but the terms involved in the associated dominant
balance, each region also has a differential operator associated to it
which captures the dominant behavior of the generator in the region as
$|(x,y)|\rightarrow \infty$. Beginning with the region adjacent to the
priming region, we propagate the priming Lyapunov function through the
adjacent region by solving an associated Poisson equation of the form
\begin{equation*}
 \left\{
\begin{aligned}
(\tilde{L}v)(x,y)&=-f(x,y)\\
v(x,y)&=g(x,y) \text{ on the boundary}\,.
  \end{aligned}
\right.
\end{equation*}
The differential operator $\tilde{L}$ is governed by the
dominant balance in the region under consideration.
%
Since it represents a dominate balanced at infinity, it is necessarily
an operator which scales homogeneously. Hence if the righthand side
and the boundary conditions are chosen to scale homogeneously at
infinity in a compatible way the solution $v$ will also scale
homogeneously at infinity.
The boundary data $g$ for the
Poisson equation is given by the dominant behavior/scaling of the priming
Lyapunov function on the boundary between the priming and adjacent
region.  The right-hand side, $f$, of the Poisson equation is chosen
to be a positive definite function which grows unboundedly and
satisfies certain scaling properties that we specify in
Section~\ref{sec:Poisson} and that are compatible with the scaling of
the region.

We iterate this procedure to construct local Lyapunov functions as
solutions to associated Poisson equations in each of the transport
regions.  Furthermore, we construct a Lyapunov function in the
diffusive region by solving a Poisson equation as well, again with the
boundary data determined by the dominant behavior of the local
Lyapunov function in the adjacent transport region.  An advantage of
this approach, therefore, is its consistency: the same procedure is
used to construct local Lyapunov functions in all but the priming
region (where the local Lyapunov function is usually straightforward to
deduce).

As we solve the sequence of Poisson equations, we encounter boundaries
without boundary data. While \textit{a priori} this could be an issue,
we will see that in our model problem, in all but the diffusive
region, the deterministic dynamics are dominant.  Hence the associated
Poisson equations are governed by first-order operators requiring only
one boundary condition. This is consistent with the idea of the
priming Lyapunov function being propagated through a sequence of
transport regions.  Again, \textit{a priori} this could lead to an
incompatibility between two different boundaries of a given region,
particularly if the relevant operator in a region is only first-order
and cannot accept generic initial data on two boundaries.  However, in
our model problem and all of the other problems we have explored,
sequences of compatible transport regions are separated from each
other by diffusive regions. Since the associated differential operator
in the diffusive region is second-order, the associated Poisson
equation produces a smooth solution even with all of its boundary data
specified.

\section{The Model Problem}
\label{sec:deterministic-model}

As our model problem, we consider essentially the same problem as
in \cite{herzogGawedzkiWehr11,herzogBirrellWehr11} and
which was suggested to us by one of the authors:
\begin{equation}\label{eq:sde}
\begin{aligned}
d X_t &= (X_t^2 - Y_t^2)dt + \sqrt{2\sigma_x} \ dW_t^{(1)}\\
d Y_t &= 2 X_t Y_t dt  +  \sqrt{2\sigma_y} \ dW_t^{(2)}
\end{aligned}
\end{equation}
with $\sigma_x \geq 0$ and $\sigma_y \geq 0$. Notice that when

$(\sigma_x,\sigma_y)=(0,0)$, the resulting deterministic equation blows
up in finite time if $x_0>0$ and $y_0=0$.  In light of this, it is
striking that for any $\sigma_y>0$, system \eqref{eq:sde} has a unique
invariant probability measure $\pi$. This was first proven in
\cite{herzogGawedzkiWehr11} and is also a consequence of one
of our main
results, which is given below in
Theorem~\ref{thm:main}. In the sequel to \cite{herzogBirrellWehr11}, the authors
prove exponential convergence to equilibrium. The principal
difficulty in both of these works was the establishment of a standard
Lyapunov function.

Let $P_t$ be the Markov
semi-group associated to the process $(X_t, Y_t)$ and defined by
\begin{align}
(P_t\phi)(x,y) = \E_{(x,y)}[\phi(X_t,Y_t)]\,.
\end{align}
Define the action of $P_t$ on a probability measure $\mu$ by $(\mu
P_t)(A)=\int P_t(x,A) \mu(dx)$ for any measurable set $A$. An
invariant probability measure  $\mu$ is any measure such that $\mu
P_t=\mu$ for all $t$.

We prove the following theorem, which is stronger than the previously cited results on the existence of a standard Lyapunov function.
\begin{theorem}
\label{thm:existsLyap}
There exists a $C^2$ function $V:\RR^2 \rightarrow (0,\infty)$  which
is a super Lyapunov function for the dynamics given by
\eqref{eq:sde}. More exactly, for any choice of $\delta\in (0,\frac25)$
the super Lyapunov function $V$ can be chosen to have an exponent
$\frac{5 \delta +5}{5 \delta +3}$ and satisfy
$c|(x,y)|^{\delta}  \leq V(x,y) \leq C|(x,y)|^{\frac52\delta +
  \frac32}$ for some positive constants $c$ and $C$.
\end{theorem}
The existence of an invariant measure $\mu$ is an easy consequence of this theorem. To determine rates of convergence to the equilibrium measure $\mu$,
we introduce the following family of weighted total variation metrics.
For $\beta > 0$ and probability measures $\mu_1$ and $\mu_2$, we
define
\begin{align}\label{eq:rhoDef}
\rho_\beta(\mu_1,\mu_2)&=\sup_{\|\phi\|_\beta\leq1}\int \phi(z)(\mu_1-\mu_2)(dz)
\end{align}
where
\begin{align*}
\|\phi\|_\beta&=\sup_z\frac{|\phi(z)|}{1+\beta V(z)}\ .
\end{align*}
Notice that $\rho_0$ is just the standard total-variation norm.

The standard Lyapunov function and supporting estimates developed in
\cite{herzogBirrellWehr11} essentially show that there exists
positive $C$ and $\eta$ so that
\begin{align*}
  \rho_1(\mu_1P_t, \mu_2P_t) \leq C e^{-\eta t} \rho_1(\mu_1,\mu_2)
\end{align*}
for any probability measures $\mu_1$ and $\mu_2$.
Using Theorem \ref{thm:existsLyap} on the existence of a super Lyapunov function, we establish the following stronger
convergence result:
\begin{theorem}\label{thm:main}
  If $\sigma_y>0$, then for any $\beta \geq 0$ there exist positive $C$
  and $\eta$ so that for all probability measures $\mu_1$ and $\mu_2$
  one has
  \begin{align*}
   \rho_\beta(\mu_1 P_t,\mu_2 P_t)\leq Ce^{-\eta t} \|\mu_1-\mu_2\|_{TV}
 \end{align*}
for all $t \geq 0$, where $| |_{TV}$ represents the total variation norm. Here the constant $C$ depends on the choice of
$\beta$ but the constant $\eta$ does not.
\end{theorem}
The strength of this result is that the dominating norm on the
right-hand side is scale and translation invariant. As we will see, when a super Lyapunov function exists, one can
usually prove a stronger result than the standard Harris-type ergodic
theorem associated to a standard Lyapunov function.

\begin{remark}
As already mentioned, the existence of an invariant measure $\mu$ follows
quickly from Theorem~\ref{thm:existsLyap} or
Theorem~\ref{thm:main}. The fact that there is only one
invariant measure is immediate from Theorem~\ref{thm:main}.
\end{remark}

One consequence of our estimates is the following
information on the unique invariant measure $\mu$. The proof of Theorem~\ref{thm:aboutMu} below is given in
Section~\ref{sec:IM}.
\begin{theorem}\label{thm:aboutMu}
As long as $\sigma_y >0$, then $\mu$ has a smooth  density with respect
to Lebesgue measure which we denote by $m(z)$.
If $\sigma_x,\sigma_y >0$, then $m(z)>0$ for all $z \in \RR^2$. If
$\sigma_x=0$ and $\sigma_y >0$, then
$m(z)=0$  if  $z=(x,y)$ with $x \geq 0$, and $m(z)>0$ if  $z=(x,y)$ with $x < 0$.
\end{theorem}

\section{Outline of Paper}
In Section~\ref{sec:superL}, we show how the existence of a super
Lyapunov function leads to a strong regularization of moments. In
Section~\ref{sec:det_eq}, we discuss further the properties of the
deterministic model problem. In Section~\ref{sec:dominantBalance}, we
perform an asymptotic analysis of the generator associated with
\eqref{eq:sde}. In Section~\ref{sec:localLyap}, we use associated
Poisson equations to construct local super Lyapunov functions in the
different regions whose boundaries are determined by the asymptotic
analysis. In Section~\ref{sec:patch}, we patch the local Lyapunov
functions together to construct the global Lyapunov function and
thereby prove Theorem~\ref{thm:existsLyap}.
In Section~\ref{sec:positivity}, we prove, under
various assumptions,the existence of a smooth transition density with various
positivity properties.  Our
approach here invokes methods from geometric control theory and
Malliavin calculus in a manner which, we hope, will be of independent
interest. In Section~\ref{sec:IM}, we transfer the smoothness and
positivity results to the invariant measure and in doing so prove
Theorem~\ref{thm:aboutMu}.
In Section \ref{sec:ergodicity}, we prove that
Theorem~\ref{thm:existsLyap}, when combined with a standard
minorization condition, implies Theorem \ref{thm:main}.
In Section~\ref{sec:elliptic2}, we show how in the
uniformly elliptic setting, namely $\sigma_x,\sigma_y >0$, the needed
minorization condition follows immediately from the strong from of
positivity which holds in that setting. In
Section~\ref{sec:minor_pos_noise-y}, we show how the weaker positivity
properties which hold when $\sigma_x=0,\sigma_y >0$ are sufficient to prove the minorization condition.
In Section~\ref{sec:conclusion}, we make a few concluding remarks.
The Appendix contains a relatively standard comparison result which we
include for completeness. It is used in Section~\ref{sec:superL} about the  Super Lyapunov Structure.

\section{Acknowledgments}
We gratefully acknowledge David Schaeffer, Tom Beale, Tom Witelski,
Jan Wehr, Martin Hairer and Charles Doering for illuminating
discussion and mathematical insight. We are especially thankful to
Charles Doering and Jan Wehr for approaching us with interesting and
stimulating model problems and for later sharing drafts of their
eventual work on these problems. Thoughts from these early discussions
and ideas developed while working on \cite{HairerMattingly2009} were
the starting point for this work.
In particular, we are indebted to Jan Wehr for bringing to us the
specific planar dynamical system that forms the core example in this
paper.  We also thank David Herzog for reading and commenting on
various drafts of this paper and Jonathan Weare for useful discussions
on numerical methods. All of the authors thank the National Science
Foundation for its support through the grants NSF-DMS-04-49910,
NSF-DMS-08-54879 (FRG)
and NSF-DMS-09-43760 (RTG). We are also indebted to SAMSI and its special year-long program on
Stochastic Dynamical Systems during which numerous motivational discussions
and meetings were held. JCM thanks the MBI for providing an
environment in which the paper could at long last be completed.

\section{Consequences of Super Lyapunov Structure}
\label{sec:superL}

We begin with a lemma, whose proof is given at the end of the section,
which is a simple translation of the bound on the generator for the
definition of a global super Lyapunov function to a bound on the
action of the semigroup. Despite its simplicity, it is nonetheless the
key to all of the enhanced results that are a consequence of the
existence of a super Lyapunov function (as opposed to merely a
standard Lyapunov function).
\begin{lemma}
\label{lemma:LyapimpliesA1}
Suppose that $V:\RR^2\to(0,\infty)$ is a super Lyapunov function for
the SDE corresponding to a Markov semi-group $P_t$.  Then for every $t>0$, there exists a
positive constant $K_t$, such that $t \mapsto K_t$ is a
continuous, monotone decreasing function on $(0,\infty)$ with $K_t
\rightarrow (2b/m)^{1/\gamma}$ as $t \rightarrow \infty$, and
\begin{equation*}
(P_tV)(z)\leq K_t \quad \text{for all} \quad z\in\RR^2 \text{ and }  t>0\,.
\end{equation*}
\end{lemma}
Recalling the definition of $\rho_\beta$ from \eqref{eq:rhoDef},
Lemma~\ref{lemma:LyapimpliesA1} implies the following result.
\begin{proposition}\label{prop:superLyapReg}
  If $V$ is a super Lyapunov function and $K_t$ is
  the constant defined in Lemma~\ref{lemma:LyapimpliesA1} then for any
  $t > 0$, $\beta>0$,  test
  function $\phi$, and probability measures $\mu_1$ and $\mu_2$, we have
  \begin{align*}
    \|P_t \phi\|_0 \leq (1+\beta K_t) \|\phi\|_\beta \qquad \text{and}\qquad
    \rho_\beta(\mu_1P_t,\mu_2P_t) \leq (1+\beta K_t)\|\mu_1-\mu_2\|_{TV}\,.
  \end{align*}
\end{proposition}
\begin{remark}\label{rem:easyDirectionInequality}
  It is clear that $\rho_0(\mu_1,\mu_2)=\|\mu_1-\mu_2\|_{TV}$ and
  furthermore if $0 \leq \alpha$, $\beta>0$ and  $K= \sup_x
  \frac{1+\alpha V(x)}{1+\beta V(x)} $, then one has $\|\phi\|_{\beta} \leq K  \|\phi\|_{\alpha}$,
  which implies $$\{ \phi :   \|\phi\|_{\alpha} \leq 1/K\} \subset \{ \phi :
  \|\phi\|_{\beta} \leq 1\}, $$ which in turn implies
  $\rho_\alpha(\mu_1,\mu_2) \leq K \rho_\beta(\mu_1,\mu_2)$. Thus as
  long as $\alpha$ and $\beta$ are both positive, the associated norms
  and metrics are equivalent. However, if one of them is zero, the
  needed inequalities only go in one direction.  Nonetheless,
  Proposition~\ref{prop:superLyapReg} allows us to use the action of $P_t$ to recover the missing inequality.
\end{remark}

\begin{proof}[Proof of Proposition~\ref{prop:superLyapReg}]
  By similar reasoning to that used in the second part of
   Remark~\ref{rem:easyDirectionInequality}, we see that if one assumes that
  $\|P_t \phi\|_0 \leq (1+\beta K_t) \|\phi\|_\beta$ for some constant $K_t$,
  then $\{ \phi : \|\phi\|_{\beta} \leq 1/(1+\beta K_t)\} \subset \{ \phi :
  \|P_t \phi\|_{0} \leq 1\}$ which then implies that
  $\rho_\beta(\mu_1P_t,\mu_2P_t) \leq (1+\beta K_t)
  \rho_0(\mu_1,\mu_2)$.  Since as noted in
  Remark~\ref{rem:easyDirectionInequality}
  $\rho_0(\mu_1,\mu_2)=\|\mu_1-\mu_2\|_{TV}$, the proof of the second
  quoted inequality is now complete provided we prove the first.

Now since $|\phi(z)| \leq \|\phi\|_\beta (1+ \beta V(z))$ for all $z$,
one has
\begin{align*}
  |(P_t \phi)(z)|  \leq \|\phi\|_\beta   \big(1+ \beta (P_t V)(z)\big)
  \leq \|\phi\|_\beta  (1+ \beta K_t)\,.
\end{align*}
Since the right-hand side is independent of $z$, we obtain the desired
result by taking the supremum over $z$.
\end{proof}

\begin{remark}
In light of Remark~\ref{rem:easyDirectionInequality} and
Proposition~\ref{prop:superLyapReg}, to prove
Theorem~\ref{thm:main} we need only prove the more standard Harris
chain-type geometric convergence result of $\rho_\beta(\mu_1 P_t,
\mu_2 P_t) \leq C \exp(-\eta t) \rho_\beta(\mu_1,
\mu_2)$ for some $\beta  >0$.
\end{remark}

\begin{proof}[Proof of Lemma~\ref{lemma:LyapimpliesA1}]
  Let $V_t=V(Z_t)$, where $Z_t$ is the solution to the SDE
  corresponding to $P_t$.  Let $L$ denote the generator associated to
  the SDE corresponding to $P_t$.  Since $V$ is a super Lyapunov
  function, there exist constants $m, b>0$ and $\gamma>1$ such that
\begin{equation*}
LV_t\leq-mV_t^\gamma+b \quad \text{for all}\quad t\geq0\,.
\end{equation*}
By Dynkin's formula,
\begin{align*}
(P_tV)(z)=\E_z[V_t]&=V(z)+\E_z\left[\int_0^t LV_s ds\right]
\leq V(z)-m\int_0^t \E_z[V_s^\gamma]ds+bt\\
& \leq V(z)-m\int_z^t\E_z[V_s]^\gamma+bt \qquad \text{by convexity.}
\end{align*}
For simplicity of notation, let $\phi_z(t)=(P_tV)(z)=\E_z[V_t]$.  Then
$\phi_z(t)$ satisfies the following differential inequality:
\begin{align*}
\phi_z'(t)&\leq-m[\phi_z(t)]^\gamma+b\\
&\leq -\frac{m}{2}[\phi_z(t)]^\gamma \quad \text{ if } \quad \phi_z(t)\geq\left(\frac{2b}{m}\right)^\frac1\gamma\,.
\end{align*}
Let $R=\left(\frac{2b}{m}\right)^\frac1\gamma$ and let
$\tau=\inf\{t>0:\phi_z(t)\leq R\}$.  Since $\phi_z'(t)<0$ if
$\phi_z(t)\geq R$, this implies that once $\phi_z(t)\leq R$,
$\phi_z(t)$ remains less than or equal to $R$ for all times afterward.
Thus, for all $t\geq\tau$, $\phi_z(t)\leq R$.  Now suppose $\psi_z(t)$
satisfies the following differential equation:
\begin{equation*}
  \left\{
    \begin{aligned}
      \psi_z'(t) &=-\frac{m}{2}[\psi_z(t)]^\gamma \quad \text{for all}\quad t\in[0,\tau]\\
      \psi_z(0) &=\phi_z(0)=V(z)\,.
    \end{aligned}
\right.
\end{equation*}
Then by Proposition~\ref{prop:comparison} in the Appendix,
$\phi_z(t)\leq\psi_z(t)$ for all $t\in[0,\tau]$.  Now the differential
equation for $\psi_(t)$ can be solved explicitly to obtain that for
all $t\in[0,\tau]$:
\begin{align*}
  \psi_z(t)
  &=\left(\frac{m(\gamma-1)t}{2}+V(z)^{-(\gamma-1)}\right)^{-\frac{1}{\gamma-1}}\leq
  \left(\frac{m(\gamma-1)t}{2}\right)^{-\frac{1}{\gamma-1}}\,.
\end{align*}
Defining the constants $K_t$ as follows
\begin{equation*}
  K_t=\max\left\{\left(\frac{2b}{m}\right)^\frac1\gamma, \left(\frac{m(\gamma-1)t}{2}\right)^{-\frac{1}{\gamma-1}}\right\}\,,
\end{equation*}
we conclude that $\phi_z(t)\leq K_t$ for all $t>0$, which completes
the proof.
\end{proof}

\section{Deterministic Equation}
\label{sec:det_eq}
To better understand the context of our results for the stochastically
perturbed system, we pause for a moment and highlight some properties
of the underlying deterministic dynamics:
\begin{equation}\label{eq:deterministic_dyn}
\begin{aligned}
\dot x_t &= x_t^2 - y_t^2\\
\dot y_t &= 2 x_t y_t\,.
\end{aligned}
\end{equation}
The trajectories of the system are shown in Figure~\ref{fig:orbits},
from which the dynamics of the system can be quickly and easily
understood.

For any initial condition $(x_0, y_0)$, the solution $(x_t, y_t)$ to
this system is given by
\begin{equation}\label{eq:exp_soln_det_dyn}
\begin{aligned}
x_t &=\frac{x_0-(x_0^2+y_0^2)t}{(1-x_0t)^2+(y_0t)^2}\\
y_t &=\frac{y_0}{(1-x_0t)^2+(y_0t)^2}\,.
\end{aligned}
\end{equation}
In particular, the system exhibits finite-time blow-up (at time
$t=\frac{1}{x_0}$) for initial conditions $(x_0, 0)$ on the positive
$x$-axis.  For all other initial conditions, the $\omega$-limit set
$\omega(x_0,y_0)$ is simply the origin, which is the unique fixed
point of the system.  We note that the origin is not reached in finite
time by any trajectory with initial condition $(x_0,y_0) \neq (0,0)$.

Now, for any choice of initial condition $(x_0, y_0)$ not on the $x$-axis,
the trajectories of the deterministic system are circles centered at the point $C(x_0, y_0)$ with radius $R(x_0,y_0)$ given as follows:
\begin{equation}\label{eq:cent_rad_det_dyn}
C(x_0,y_0)=\Bigl(0, \frac{x_0^2+y_0^2}{2y_0}\Bigr),\qquad R(x_0,y_0)=\frac{x_0^2 + y_0^2}{2|y_0|}\,.
\end{equation}

Furthermore, for all choices of initial conditions $(x_0,y_0)$ not on the positive $x$-axis,
the time to return to a fixed ball of radius $R$ about the origin is
uniformly bounded by $\frac2R$.  In Section \ref{sec:minor_pos_noise-y}, we employ this uniform bound to prove a positivity and minorization condition on the transition density for the stochastically-perturbed system.

\begin{figure}
\centering
 \begin{tikzpicture}[scale=.4]
%
\draw[color=white,thick] (5,5)--(5,-5)--(-5,-5)--(-5,5) --(5,5);
\draw[thick,color=gray] (-5.5,0) -- (5.5,0) node[below right] {$x$};
\draw[thick,color=gray] (0,5.5) node[above right] {$y$} --    (0,-5.5);
\draw[>->,color=red,thick,opacity=.75]
plot[id=T1,domain=-2.5:2.5,smooth,samples=150] function {+2.5+sqrt(
 (2.5)*(2.5) - x*x ) } ;
\draw[-,color=red,thick,opacity=.75]
plot[id=T2,domain=-2.5:2.5,samples=150] function {+2.5-sqrt((2.5)*(2.5) - x*x)};
\draw[<-<,color=red,thick,opacity=.75] plot[id=T1a,domain=-2.5:2.5,smooth,samples=70] function{-2.5+sqrt( (2.5)*(2.5) - x*x)};
\draw[-,color=red,thick,opacity=.75] plot[id=T2a,domain=-2.5:2.5,smooth,samples=70] function{-2.5-sqrt((2.5)*(2.5) - x*x)};
\draw[-,color=red,thick,opacity=.75] plot[id=B1,domain=-5:5,smooth,samples=150,thick] function{+5+sqrt( (5)*(5) - x*x)};
\draw[>->,color=red,thick,opacity=.75] plot[id=B2,domain=-5:5,smooth,samples=150] function{+5-sqrt( (5)*(5) - x*x)};
\draw[<-,color=red,thick,opacity=.75] (-4,0) -- (-8,0);
\draw[color=red,thick,opacity=.75] (0,0) -- (-4,0);
\draw[color=red,thick,opacity=.75] (4,0) -- (8,0);
\draw[->,color=red,thick,opacity=.75] (0,0) -- (4,0);
\draw[-,color=red,thick,opacity=.75] plot[id=B1a,domain=-5:5,smooth,samples=150] function{-5+sqrt( (5)*(5) - x*x)};
\draw[<-<,color=red,thick,opacity=.75] plot[id=B2a,domain=-5:5,smooth,samples=150] function{-5-sqrt((5)*(5) - x*x)};
\end{tikzpicture}
\caption{A number of representative
 orbits of  the deterministic dynamics governed by
 \eqref{eq:deterministic_dyn}.}
\label{fig:orbits}
\end{figure}
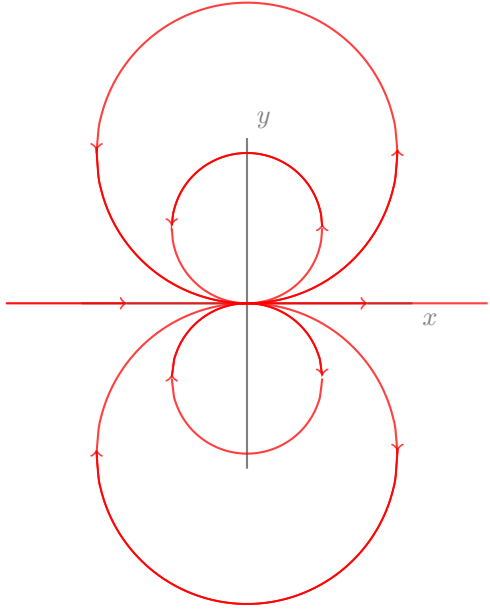

\section{Dominant Balances of Generator}
\label{sec:dominantBalance}

We now begin the program laid out in
Section~\ref{sec:construction}. We begin by considering the dominant
operators in various regions of the state space.

Associated to the SDE \eqref{eq:sde} is the generator $L$ defined by
\begin{align}\label{eq:generator}
  L &= (x^2-y^2)\partial_x + 2 x y \partial_y + \sigma_x\partial_{xx}+
  \sigma_y\partial_{yy}\,.
\end{align}
In order to prove that the addition of noise arrests the blow-up on
the $x$-axis sufficiently to produce an invariant probability measure,
we need to understand the behavior of the dynamics at infinity.  There
are many different routes to infinity and we now consider the various
possible dominant balances associated with different routes.

To help identify the relevant scaling, consider the behavior of $L$
under the scaling map $(x,y)\mapsto (\ell x,\ell^p y)$ which produces
\begin{align*}
  \ell\, x^2 \partial_x-\ell^{2p-1}\,y^2\partial_x +\ell\, 2 x y \partial_y +\ell^{-2}\, \sigma_x\partial_{xx}+\ell^{-2p}\,\sigma_y\partial_{yy}\,.
\end{align*}
If $p=1$ the first three terms balance and dominate the remaining
terms as $\ell \rightarrow \infty$. If $p >1$ then the second term
dominates. If $p=-\frac12$ then the first, third and fifth balance and
dominate all other terms as $\ell \rightarrow \infty$. These balances
cover all of the routes to infinity except for those which approach or
rest on the $y$-axis and identify $p=-1/2$ as a critical
scaling. (The  routes near  the $y$-axis are captured by $p=-1/2$ and
$\ell \rightarrow 0$ but these will not play an important role in our
analysis.) 

If  $|x|y^2<\infty$ as  $|(x,y)|\rightarrow\infty$ with $x>a>0$, the
dominant part of $L$ is contained in
\begin{align}
\label{eq:A}
\A&= x^2\partial_x + 2 x y \partial_y + \sigma_y\partial_{yy}\,.
\end{align}
If $|x|y^2\rightarrow0$ as $|(x,y)|\rightarrow\infty$ with $x>a>0$, then the dominant
part is only $\partial_{yy}$. Notice that $\partial_{yy}$ is contained
in $\A$, so we can still choose to use $\A$ this region. In all other
relevant cases as $|(x,y)|\rightarrow\infty$, the
dominant part of $L$ is contained in
\begin{align*}
 \T = (x^2- y^2)\partial_x + 2 x y \partial_y\,.
\end{align*}
We have neglected the term $\sigma_x\partial_{xx}$ in the operator $T$
which scaling analysis suggests might be relevant in neighborhood of
the $y$-axis. However its inclusion does not  qualitatively change the
behavior in a neighborhood of
the $y$-axis. The same can not be said of the term
$\sigma_y\partial_{yy}$ in a  neighborhood of
the $x$-axis. 

\subsection{Scaling}
\label{sec:scaling}
To better understand the structure of the solutions in the various
regimes, we investigate the scaling properties of the various
operators introduced in the previous section. We introduce the scaling transformations
\begin{align*}
S_\ell^{(1)}\colon (x,y) &\mapsto (\ell x,\ell^{-\frac12} y) \quad\text{and}\quad
S_\ell^{(2)}\colon (x,y) \mapsto (\ell x,\ell y)\,.
\end{align*}
Observe that operator $\A$ scales homogeneously under the scaling $S_\ell^{(1)}$, while
the operator $\T$ scales homogeneously under the scaling $S_\ell^{(2)}$.  We
would also like the operator $\T$ to scale homogeneously under the scaling
$S_\ell^{(1)}$; however, this does not hold for all of the terms in $T$. We
remedy this by introducing a non-negative parameter $\lambda$ and defining the family of
operators
\begin{align}
 \label{eq:T}
 \T_\lambda = (x^2 - \lambda y^2)\partial_x + 2 x y \partial_y
\end{align}
and extending the definition of the scaling operators by
\begin{align*}
S_\ell^{(1)}\colon (x,y,\lambda) &\mapsto (\ell x,\ell^{-\frac12} y,\ell^3
\lambda)\quad\text{and}\quad
S_\ell^{(2)}\colon (x,y, \lambda) \mapsto (\ell x,\ell y, \lambda)\,.
\end{align*}
Now $\T_\lambda$ scales homogeneously under the scaling map
$S_\ell^{(1)}$ and $A$ remains invariant under $S_\ell^{(2)}$. This
gambit of introducing an extra parameter to produce a homogeneous scaling was also used in a similar way in
\cite{CookeMattinglyMcKinleySchmidler2011}.

Given a function $\phi:\mathcal{R} \times [0,\infty) \rightarrow \RR$, where $\mathcal{R}\subset\RR^2$, we
say that $\phi$ scales homogeneously under the scaling $S_\ell^{(i)}$ if
$\phi \circ S_\ell^{(i)} = \ell^\delta \phi$ for some $\delta$.  In this case, we say that
$\phi$ scales like $\ell^\delta$ under the $i$-th scaling. We write this compactly as $\phi \overset{i}{\sim} \ell^\delta$.
\begin{proposition}\label{prop:scaling}
If $\phi \simI \ell^\delta$  then $\partial_x \phi  \simI
\ell^{\delta-1}$ and $\partial_y \phi  \simI
\ell^{\delta+\frac12}$. Similarly, if  $\phi \simII \ell^\delta$ then $\partial_x \phi  \simII
\ell^{\delta-1}$ and $\partial_y \phi  \simII
\ell^{\delta-1}$. In both cases, if one side is infinite, then so is
the other.
\end{proposition}
\begin{proof}[Proof of Proposition~\ref{prop:scaling}]
We only show one case; all others follow similarly. If $\phi
\simI \ell^\delta$, then $\phi(\ell x , \ell^{-\frac12}
y,\ell^3\lambda)=\ell^\delta \phi(x,y,\lambda)$. Differentiating in
$x$, we obtain
\begin{align*}
  \ell (\partial_x\phi)(\ell x , \ell^{-\frac12}
  y,\ell^3\lambda)=\ell^\delta(\partial_x\phi)(x,y,\lambda)\,.
\end{align*}
Dividing through by $\ell$, we conclude that  $\partial_x \phi  \simI
\ell^{\delta-1}$.
\end{proof}

In the next section, we decompose the plane into regions
where the different dominant balances hold. These regions are defined by
boundary curves which are well-behaved under one or both of the
scalings. To facilitate the construction of these regions, given
$x_0>0, y_0>0$, $\lambda\geq0$ and $p \in \RR$, we define
the following ``elementary'' regions:
\begin{align}
\regionO(x_0,y_0,\lambda)&=\Big\{ \frac{x^2 + \lambda y^2}{|y|} \geq
\frac{x_0^2 + \lambda y_0^2}{y_0} \Big\} \notag\\
\label{eq:P}
\regionP^\pm_p(x_0,y_0) &=\{  \pm x \geq x_0 , |x|^p |y| \leq x_0^p y_0 \}\,.
\end{align}
Observe that for any $\ell >0$, we have to following scaling relations
\begin{xalignat*}{2}
   S_\ell^{(1)}( \regionP^\pm_p(x_0,y_0))&= \regionP^\pm_p(\ell
  x_0,\ell^{-\frac{1}2}y_0),&  S_\ell^{(2)}( \regionP^\pm_p(x_0,y_0))&= \regionP^\pm_p(\ell
  x_0,\ell y_0), \\
 S_\ell^{(1)}( \regionO(x_0,y_0,\lambda))&=
  \regionO(\ell x_0, \ell^{-\frac12}y_0, \ell^3\lambda),
  &   S_\ell^{(2)}( \regionO(x_0,y_0,\lambda))&=
  \regionO(\ell x_0, \ell y_0, \lambda),
\end{xalignat*}
and lastly $\regionP^\pm_{p}(\ell x_0,\ell^{-p} y_0) \subset \regionP^\pm_{p}(x_0,y_0)$ for $\ell >1$ .\\

\section{Construction of Local Lyapunov Functions}
\label{sec:localLyap}

Based on the discussion in the previous section, we will  divide the plane into three regions
$\mathcal{R}_i(\alpha)$, where $\alpha$ is a positive parameter that
we specify later.  As described in Section \ref{sec:construction}, we
call these regions the ``priming," ``transport," and ``diffusive"
regions, respectively.  We now describe the placement of these various regions which
are indicated pictorially in  Figure~\ref{fig:regions}.

Our priming region, $\mathcal{R}_1(\alpha)$,
is a subset of the left-half plane, and here there exists a very
natural Lyapunov function, because in this region, the deterministic
drift is directed toward the origin.  On the other hand, the diffusive
region, $\mathcal{R}_3(\alpha)$, is a funnel-like region around the
positive $x$-axis where there is finite-time blow-up in the
deterministic setting. Demonstrating the existence of a local
Lyapunov function in the diffusive region is a key piece in proving
noise-induced stabilization in our model problem.  The transport
region $\mathcal{R}_2(\alpha)$ is governed primarily by deterministic
transport from the diffusive region to the priming region. In this
section, we focus on the construction of a local Lyapunov function in
each of these three regions.

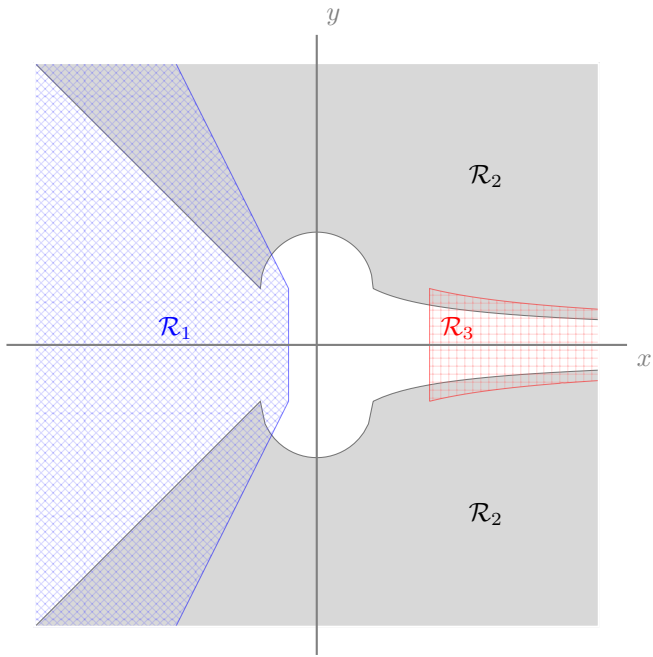
\begin{figure}
  \centering
 \begin{tikzpicture}[scale=.75]
\draw[color=black,thin,fill=black!30,opacity=.5]	plot[id=rightTopBoundary,domain=1:5,samples=70]
function{sqrt(1/(x))} -- (5,.4472) -- (5,5) -- (-5,5) -- (-1,1) -- 	plot[id=centerTop,domain=-1:1,samples=70] function{1+ sqrt(1-x**2)};
\draw[color=black,thin,fill=black!30,opacity=.5]	plot[id=rightBottomBoundary,domain=1:5]
function{-sqrt(1/(x))} -- (5,-.4472) -- (5,-5) -- (-5,-5) -- (-1,-1)
-- 	plot[id=centerBottom,domain=-1:1] function{-1- sqrt(1-x**2)};
\node at (3,3) {$\mathcal{R}_2$};
\node at (3,-3) {$\mathcal{R}_2$};
  \draw[color=blue,thin,pattern=crosshatch,pattern color
  = blue!60,opacity=.5]
  plot[id=leftTopBoundaryA,domain=-2.5:-1]
  (-0.5,1) --(-2.5,5) -- (-5,5) -- (-5,-5) -- (-2.5,-5) -- (-0.5,-1) -- (-0.5,1) ;
\node[color=blue] at (-2.5,.3) {$\mathcal{R}_1$};
   \draw[color=red,thin,pattern=grid,pattern color
  =red!60,opacity=.5,samples=70] 	(2,-1.0)--(2,1.0) --
   plot[id=rightTopBoundaryA,domain=2:5] function{sqrt(1/(0.5*x))}
   -- (5,.632455) -- (5,-.632455)--  plot[id=rightBottomBoundaryA,domain=0:3,samples=70]
(5.0-\x,{-sqrt(1/(0.5*(5.0-\x)))});
\node[color=red] at (2.5,.3) {$\mathcal{R}_3$};

\draw[color=white,thick] (5,5)--(5,-5)--(-5,-5)--(-5,5) --(5,5);
\draw[thick,color=gray] (-5.5,0) -- (5.5,0) node[below right] {$x$};
\draw[,thick,color=gray] (0,5.5) node[above right] {$y$} --
    (0,-5.5);
\end{tikzpicture}
\caption{The different regions in which local Lyapunov functions are
constructed. $\mathcal{R}_1$ is the priming region. The two regions
labeled $\mathcal{R}_2$  are transport regions. And
$\mathcal{R}_3$
is the diffusive region which connects the two transport regions in
which information is propagating in different directions.}
\label{fig:regions}
\end{figure}

\subsection{The Priming Region}\label{sec:priming}
When looking for a priming Lyapunov function, it is natural to consider
the norm to some power. In this specific example, we expect
the norm to some power to be a Lyapunov function in the left-half plane
since the drift vector field points at least partially towards the
origin; see Figure~\ref{fig:orbits}.

For $\delta>0$, we define $v_1(x,y)=(x^2+y^2)^\frac\delta2$ and observe
that
\begin{equation}\label{eq:wFirtTry}
\begin{aligned}
Lv_1(x,y)=&\delta
  x(x^2+y^2)^{\frac\delta2} +\delta (\frac{\delta}{2} -1) (x^2+y^2)^{\frac{\delta}{2}-2}(2\sigma_x x^2 + 2 \sigma_y y^2)\\
&+(\sigma_x+\sigma_y)\delta(x^2+y^2)^{\frac{\delta}{2}-1}\,.
\end{aligned}
\end{equation}
In particular, if $(x,y) \in
\regionP^-_{-1}(\frac\alpha2,1)$, we get that
\begin{align*}
  (Lv_1)(x,y)\leq&
  -\frac{\alpha\delta}{\sqrt{\alpha^2+4}}(x^2+y^2)^{\frac{\delta+1}{2}}+\delta (\delta -2)\frac{\sigma_x x^2 +  \sigma_y y^2}{(x^2+y^2)^{2-\frac{\delta}{2}}}+\delta\frac{\sigma_x+\sigma_y}{(x^2+y^2)^{1-\frac{\delta}{2}}}\\
  =&-\frac{\alpha\delta}{\sqrt{\alpha^2+4}}(x^2+y^2)^{\frac{\delta+1}{2}}\\
  &\qquad\times \left[1- \frac{\sqrt{\alpha^2 + 4}}{\alpha}
            \left(\frac{\left[\delta
          -2\right](\sigma_x x^2+\sigma_y y^2)}{(x^2
        +y^2)^{\frac52}}+\frac{\sigma_x+\sigma_y}{(x^2+y^2)^\frac32}\right)\right]\,.
\end{align*}
This implies that for any $\delta >0$ and $\alpha >0$, there exists an $R_1$ sufficiently
large  so that if $|(x,y)| > R_1$, then the term in the
square brackets is greater than $\frac12$. Hence $v_1$ is a super Lyapunov
function in the region
\begin{align*}
 \mathcal{R}_1(\alpha)\eqdef\regionP^-_{-1}(\frac\alpha2,1)
\end{align*}
with exponent $\frac{\delta +1}{\delta}$. As we will see later,
we will have to restrict $\delta$ to the interval $(0,\frac25)$, and this automatically implies that $\delta  \in (0,2)$. In turn, this guarantees that
$\delta-2 <0$ and that the term in the square brackets above is greater
that $\frac12$ provided
\begin{align*}
 \frac{\sigma_x+\sigma_y}{(x^2+y^2)^{\frac32}} <\frac12\frac {\alpha}{\sqrt{\alpha^2 + 4}}\,.
\end{align*}
We formalize
this observation in the following proposition.
\begin{proposition}\label{prop:superv1} For any $\alpha >0$ and
  $\delta \in (0,2)$, if
  $(x,y)\in \mathcal{R}_1(\alpha) \textrm{ with } |(x,y)|\geq R_1,$ then $v_1$
  satisfies
\begin{align*}
  (L v_1)(x,y) \leq - m_1 \ v_1^{\gamma_1}(x,y)
\end{align*}
where $m_1=\frac{\alpha\delta}{2\sqrt{\alpha^2+4}}>0$,
$\gamma_1=\frac{\delta+1}{\delta}>1$, $R_1=\left[2\left(\sigma_x +\sigma_y\right)\frac{\sqrt{\alpha^2+4}}{\alpha} \right]^{\frac13}$.
\end{proposition}

Our choice of the region $\mathcal{R}_1(\alpha)$ is motivated by the
following.  From \eqref{eq:wFirtTry}, it is clear that we need to
define a region in the negative half-plane bounded away from the
$y$-axis.  Furthermore, in order to guarantee that $v_1$ is super
Lyapunov, we need to ensure a region in which $|(x,y)| \rightarrow
\infty$ implies $|x| \rightarrow
\infty$.
Note that $\mathcal{R}_1(\alpha)$ is a subset of the left half-plane,
in which the dominant dynamics at infinity are given by $T$ and hence
the relevant scaling transformation is  $S_\ell^{(2)}$. For this
reason, it is desirable to define the boundary of the region so that it
behaves well under $S_\ell^{(2)}$.  From the
previous section, we see that
\begin{align*}
    S_\ell^{(2)}( \regionP^-_{-1}(x_0,y_0))&= \regionP^-_{-1}(\ell
  x_0,\ell y_0) \subset \regionP^-_{-1}(x_0,y_0)\quad\text{for}\quad \ell>1
\end{align*}
which motivates our choice of $\mathcal{R}_1(\alpha)$ and the shape of
its boundary in particular.

\subsection{Decomposition of Remainder of Plane}
We will now propagate the priming Lyapunov function through a sequence
of regions until all of the routes to infinity are covered.

As mentioned above, near the boundary of $\mathcal{R}_1(\alpha)$ and away from the $x$-axis, the
operator $T$ is dominant. This holds true until one enters the region defined by the curves $xy^2=c$ where $c$
is a sufficiently large positive constant and $x>0$ is sufficiently
large.  At this point, the dominant balance changes and the operator
$A$ becomes dominant. Hence we define the transport region,
$\mathcal{R}_2(\alpha)$, with one boundary inside the region
$\mathcal{R}_1(\alpha)$ which is invariant under the scaling
$S^{(2)}_\ell$, and one boundary which is defined by the curve
$|x|y^2=c$ for some constant.  As we make precise in the definition
below, we will choose $c=\alpha$.

We set $\mathcal{R}_2(\alpha) \eqdef \mathcal{R}_2(\alpha,1)$ where for
$\alpha,\lambda\geq0$, we define
\begin{align*}
  \mathcal{R}_2(\alpha,\lambda)&\eqdef
  \overline{\regionP^+_{\frac12}(\alpha\sqrt{\lambda},1 )^c \cap
    \regionO(\alpha\sqrt{\lambda},1,\lambda) \cap
    \regionP^-_{-1}(\alpha\sqrt{\lambda},1)^c}\,.
\end{align*}

Now, observe that outside of $\mathcal{R}_1(\alpha)\cup
\mathcal{R}_2(\alpha)$ all of the routes to infinity have
$|x|y^2<\infty$. Hence the operator $\A$ is dominant in this entire
region and we do not need to further subdivide the remainder of the
plane.  To define $\mathcal{R}_3(\alpha)$, recall that we need
nontrivial overlap with the transport region
$\mathcal{R}_2(\alpha)$. Hence we again chose a boundary curve of
the form $|x|y^2=c$ but with $c >\alpha$. In particular, we define
$\mathcal{R}_3(\alpha) \eqdef \regionP^+_{\frac12}(2\alpha,1)$.  Note that
$\mathcal{R}_3(\alpha)$ is the diffusive region: the diffusion term in
the operator $A$ is critical to the stabilization of the process here.

In summary, for each $\alpha>0$, we have defined three regions
\begin{equation*}
   \begin{aligned}
     \text{Priming Region: }& \mathcal{R}_1(\alpha) = \regionP^-_{-1}(\tfrac\alpha2,1)  \\
     \text{Transport Region: }& \mathcal{R}_2(\alpha) = \mathcal{R}_2(\alpha,1)\\
     \text{Diffusive Region: }& \mathcal{R}_3(\alpha) =
     \regionP^+_{\frac12}(2\alpha,1)\,.
   \end{aligned}
\end{equation*}
Notice that $\mathcal{R}_1(\alpha)\cap\mathcal{R}_2(\alpha)$ and
$\mathcal{R}_2(\alpha)\cap\mathcal{R}_3(\alpha)$ are nonempty and that
$\RR^2\backslash(\mathcal{R}_1(\alpha)\cup\mathcal{R}_2(\alpha)
\cup\mathcal{R}_3(\alpha))$ is a bounded set.  We construct a local super Lyapunov
function in each of the three regions and then smoothly patch them
together to form one global super Lyapunov function on the entire plane.

\subsection{The Associated Poisson Equations}\label{sec:Poisson}
We now propagate the priming Lyapunov function $v_1$ which is defined
in $\mathcal{R}_1(\alpha)$ to the other regions by solving a succession of
Poisson equations.  Throughout most of our construction, $\alpha$ will
remain a free parameter; we specify $\alpha$ later to ensure a number of
necessary estimates. We begin with the transport region
$\mathcal{R}_2(\alpha)$.

\subsubsection {The Transport Region $\mathcal{R}_2(\alpha)$}
\label{sec:outer-domain}
For $\delta>0$  and $\alpha>0$, we define $v_2(x,y)$ as the solution to the following Poisson equation:
\begin{equation}
 \label{eq:poissonUBasic}
 \left\{
   \begin{aligned}
     (\T v_2)(x,y) &= -\left(\frac{x^2+ y^2}{|y|}\right)^{\delta+1}&&\text{on $\mathcal{R}_2(\alpha)$}\\
     v_2(x,y) &= v_1(x,y) &&\text{on }\partial B_1(\alpha)
   \end{aligned}
 \right.
\end{equation}
where $\partial B_1(\alpha)=\displaystyle \left\{ x\leq-\alpha, |y|= \frac1\alpha |x| \right\}$.

The rationale for this is as follows. We wish to propagate the priming
Lyapunov function through the region $\mathcal{R}_2(\alpha)$, so we
need to take it as the boundary condition.  Since the operator $T$
represents one of the dominate balances, it necessarily scales
homogeneously. In this case, $T$ scales like $\ell^1$ under the
scaling transformation $S^{(2)}_\ell$. Hence if $v_2$ is to scale
homogeneously under $S^{(2)}_\ell$ as $\ell^p$ for some power $p$ then
the righthand side must scale like $\ell^{p+1}$ and the boundary
conditions must scale like $\ell^p$ both under $S^{(2)}_\ell$. (Notice
the boundary $\partial B_1$ is invariant under  $S^{(2)}_\ell$.)

Notice that our choice of right-hand side scales as $\ell^{\delta+1}$
and the boundary conditions scale as would be consistent with a
solution which scales like $\ell^\delta$ under  $S^{(2)}_\ell$. The
form of the boundary conditions are dictated by our choice of
$v_1$. The exact from of the righthand side was chosen so that it was constant along trajectories
of the limiting dynamics in $\mathcal{R}_2(\alpha)$ which are the
characteristics of $T$.

\subsubsection{The Diffusive Region $\mathcal{R}_3(\alpha)$}
\label{sec:inner-domain}
For $\delta>0$ and $\alpha>0$, we define $v_3(x,y)$ by the following
Poisson equation
\begin{equation}\label{eq:innerPoisson}
\left\{
  \begin{aligned}
(\A v_3)(x,y) &= \displaystyle  - c_1 x^{\hat{\delta}+1} && \text{ on $\mathcal{R}_3(\alpha) $}\\
v_3(x,y) &=  c_2 x^{\hat{\delta}} && \text{ on } \partial B_2(\alpha)
  \end{aligned}
\right.
\end{equation}
where $\partial B_2(\alpha)=\displaystyle \{ x\geq\alpha, xy^2=
2\alpha \}$,
\begin{align}
  \label{deltaHat}
  \hat{\delta}=\frac52\delta+\frac32
\end{align}
and $c_1,c_2>0$ are constants which will be chosen later.  We remark
that the values of $c_1$ and $c_2$ do not affect the local super
Lyapunov property of $v_3$, but rather are chosen in order to
facilitate the patching of the local super Lyapunov functions into one
global super Lyapunov function in Section \ref{sec:patch}.  As before,
we have chosen a right-hand side which is negative definite, scales homogeneously under
the appropriate scaling, namely $S^{(1)}_\ell$, and has unbounded growth in the region.  We use a constant
multiple of $x^{\hat{\delta}}$ as the boundary condition rather than
the function $v_2$ from the neighboring region because we want a
function which scales homogeneously under $S^{(1)}_\ell$. However,
$x^{\hat{\delta}}$ is in fact the asymptotic behavior (up to a
constant multiple) of $v_2(x,y)$ as $|(x,y)| \rightarrow \infty$ on
the specified boundary.

In Section~\ref{sec:AssembleLyapunov}, we verify that $v_2$ and $v_3$
are super Lyapunov functions in the regions in which they are
defined. However, we first establish a number of preliminary results.

\subsection{Existence of Solutions and Their Properties}
The scaling properties of the solutions to the above Poisson equations
are one of main tools we use to show that they are local Lyapunov
functions. This is because, with one exception, points at infinity in
a given region can be scaled back to points in the same region by the
scaling transformation under which the associated differential
operator is homogeneous. As we discuss below, the exception is the
subregion of $\mathcal{R}_2(\alpha)$ which lies near the boundary of
$\mathcal{R}_3(\alpha)$.

\subsubsection{Properties of the Solution in the Transport Region}

Care must be taken when scaling the points in the subregion of
$\mathcal{R}_2(\alpha)$ which lie close to the boundary of
$\mathcal{R}_3(\alpha)$.  The points in this region naturally scale
with $S_\ell^{(1)}$ while the operator $\T$ which is associated to
$\mathcal{R}_2(\alpha)$ scales homogeneously under
$S_\ell^{(2)}$. This issue was also addressed in
Section~\ref{sec:scaling} where we introduced the parameter $\lambda$
to generate a family of operators $T_{\lambda}$ which scale
homogeneously with $S_\ell^{(1)}$.

With this mind, it is natural to introduce the
function $v_2(x,y,\lambda)$ which, for a given $\lambda \in (0,1]$, solves the
following family of auxiliary Poisson equations in $\mathcal{R}_2(\alpha, \lambda)$:
\begin{equation}
 \label{eq:poissonU}
 \left\{
   \begin{aligned}
     (\T_\lambda v_2)(x,y,\lambda) &= -h(x,y,\lambda)\quad \text{ on $\mathcal{R}_2(\alpha,\lambda)$}\\
     v_2(x,y,\lambda) &= f(x,y,\lambda)\qquad \text{ on } \partial
     B_1(\alpha \sqrt{\lambda})
   \end{aligned}
 \right.
\end{equation}
where we define
\begin{xalignat*}{2}
 h(x,y,\lambda)&=\left(\frac{x^2+\lambda y^2}{|y|}\right)^{\delta+1} &
 f(x,y,\lambda)&=\lambda^{\frac{\delta+1}2}(x^2+\lambda y^2)^{\frac\delta2}\,.
\end{xalignat*}
For ease of notation, we write
\begin{equation*}
h(x,y)=h(x,y,1) \quad\text{and} \quad f(x,y)=f(x,y,1)\,.
\end{equation*}
Notice that $h \simI  \ell ^{\hat{\delta}+1}$,  $f \simI \ell^{\hat{\delta}}$,
  $h \simII  \ell ^{\delta+1}$, and  $f \simII \ell^{\delta}$ where
  $\hat \delta$ was defined in \eqref{deltaHat}. Also
  observe that $v_2(x,y,1)$ coincides with the $v_2(x,y)$ defined by
\eqref{eq:poissonUBasic}.

\subsubsection{Properties of the Solution in the Diffusive Region}

The dynamics associated to the operator $\A$, which is dominant in
$\mathcal{R}_3(\alpha)$, should be understood as having one diffusive
direction and one deterministic direction which is uncoupled from the
diffusion and acts as the ``clock'' of the diffusion. To see this, observe that
$\A$ is the operator associated to the system of SDEs given by
\begin{equation}\label{eq:sdehat}
\begin{aligned}
 d \hat{X}_t &= \hat{X}_t^2dt & \hat{X}_0&=x\\
 d \hat{Y}_t &= 2 \hat{X}_t \hat{Y}_t dt  +  \sqrt{2\sigma_y} \ dW_t  &\hat{Y}_0&=y\,.
\end{aligned}
\end{equation}
Now, let $(\hat{X}_0, \hat{Y}_0)=(x,y)$ lie in $\mathcal{R}_3(\alpha)$ and define
$\hat\tau=\inf\{t>0:(\hat{X}_t,\hat{Y}_t)\in\partial B_2(\alpha)\}$.
Then $v_3(x,y)$, which was defined in \eqref{eq:innerPoisson}, can be represented probabilistically as
\begin{align}
  v_3(x,y)
  &=c_2\E_{(x,y)}\bigl[\hat{X}_{\hat{\tau}}^{\hat{\delta}}\bigr]
  +c_1\E_{(x,y)}\Bigl[\int_0^{\hat{\tau}}
    \hat{X}_s^{\hat{\delta}+1}ds\Bigr]\notag \\
  &=\bigl(\tfrac{c_1}{\hat{\delta}}+c_2\bigr)\E_{(x,y)}\bigl[\hat{X}_{\hat{\tau}}^{\hat{\delta}}\bigr]-\tfrac{c_1}{\hat{\delta}}x^{\hat{\delta}} \label{eq:vStopping1}
\end{align}
provided that, first, the expectation is finite; and second, that the
right-hand side of equation \eqref{eq:vStopping1} depends in a $C^2$
fashion on $(x,y)\in \mathcal{R}_3(\alpha)$. Both of these facts will
follow from Proposition~\ref{prop:moments}, which we present below,
and are made formal in Proposition~\ref{prop:mainPoisson}, which appears in
the next section.

Since $\hat X_t$ is deterministic, this representation of $v_3$ amounts
to a deterministic function of $\hat \tau$. To better understand the
properties of $\hat \tau$, we introduce the time change
$T(t)=\int_0^t\hat{X}_sds=-\ln|1-xt|$ and the process $Z_{T(t)}=
\hat{X}_t^{\frac12} \, \hat{Y}_t$. Due to the scaling of the
boundary of $\mathcal{R}_3(\alpha)$, if we define  $\tau=\inf\{T>0:|Z_T|\geq\sqrt{2\alpha}\}$
then $\displaystyle \hat \tau=\tfrac{1}{x}(1-e^{-\tau})$,  $\hat{X}_t=x e^{T(t)}$, and $Z_T$
satisfies the SDE
\begin{equation}\label{eq:Zt_defn}
dZ_T=\frac52 Z_TdT+\sqrt{2\sigma_y} \ dW_T\,, \qquad Z_0=x^{\frac12}y\,.
\end{equation}
Since $Z_T$ is the solution to a Gaussian SDE, the following proposition follows easily.
\begin{proposition}\label{prop:moments}
For $\hat{\delta}<\frac52$ and $(x,y) \in \mathcal{R}_3(\alpha)$,
$\E_{(x,y)}\bigl[e^{\hat{\delta}\tau}\bigr]<\infty$ and the map $(x,y) \mapsto
 \E_{(x,y)}\bigl[e^{\hat{\delta}\tau}\bigr]$ is $C^2$.
\end{proposition}
\begin{proof}[Proof of Proposition~\ref{prop:moments}]
To see the finiteness of the expectation, observe that
\begin{align*}
  \PP_{(x,y)}(e^{\hat{\delta}\tau}>s)
  &=\PP_{(x,y)}\Bigl(\sup_{0\leq T\leq \frac{\ln{s}}{\hat{\delta}}} |Z_T|<\sqrt{2\alpha}\Bigr)
  \leq\PP_{(x,y)}\left(|Z_{\frac{1}{\hat{\delta}}\ln{s}}|<\sqrt{2\alpha}\right)\\
  &=\PP\Bigl(\Big{|}\sqrt{2\sigma_y} \ s^{\frac{5}{2\hat{\delta}}}\int_0^{\frac{1}{\hat{\delta}}\ln{s}} e^{-\frac52 r} dW_r\Big{|}<\sqrt{2\alpha}\Bigr)
  \leq\Bigl(\frac{10\alpha}{\sigma_y\pi(s^{5/\hat{\delta}}-1)}\Bigr)^{\frac12}\,.
\end{align*}
Hence for $\hat{\delta}<\frac52$, this decays sufficiently rapidly in order to
guarantee that $\E_{(x,y)}\bigl[e^{\hat{\delta}\tau}\bigr]$ is finite. The
 continuity properties now follow from the continuity properties of
 $\tau$. Specifically, $\E_{(x,y)}\bigl[e^{\hat{\delta}\tau}\bigr]=g(\sqrt{x} \,
 y)$ where $g(z)$
  solves the following ordinary differential equation
\begin{align}\label{g}
  \begin{cases}
    \sigma_y g''(z)+\frac52zg'(z)+\hat{\delta}g(z)=0 & \text{ for } g\in(-\sqrt{2\alpha},\sqrt{2\alpha})\\
    g(\sqrt{2\alpha})=g(-\sqrt{2\alpha})=1\,.
  \end{cases}
\end{align}
Since by standard results on the regularity of ODEs, $g(z)\in
C^2([-\sqrt{2\alpha},\sqrt{2\alpha}])$, we conclude that $\E_{(x,y)}\bigl[e^{\hat{\delta}\tau}\bigr]=g(\sqrt{x}\ y)\in
C^2(\mathcal{R}_3(\alpha))$ as desired.
\end{proof}
We remark that this proposition imposes a further restriction on the
size of the parameter $\delta$, which previously was only required to
be positive.  Observe that in light of \eqref{deltaHat} the
requirement that $\hat{\delta}<\frac52$ forces $\delta \in (0, \frac25)$.

\subsubsection{Principal Result on Existence and Scaling of Solutions}

We consolidate these observations and now state and prove our principal existence and scaling result.
\begin{proposition}\label{prop:mainPoisson} For every $\delta \in (0,\frac25)$, there exists a strictly positive $C^2$ function $v_3\colon\mathcal{R}_3(\alpha) \rightarrow (0,\infty)$
  which solves \eqref{eq:innerPoisson}. For every $\lambda \in (0,1]$,
  there exists a strictly positive $C^2$ function $v_2\colon
  \mathcal{R}_2(\alpha,\lambda)\rightarrow (0,\infty)$ which solves
  \eqref{eq:poissonU}.  In addition, $v_2 \simI \ell
  ^{\hat{\delta}}$, $v_2\simII \ell ^{\delta}$, $v_3 \simI
  \ell^{\hat{\delta}}$ and $(x,y,\lambda) \mapsto v_2(x,y,\lambda)$ is continuous
on $\mathcal{R}_2^*(\alpha) \times [0,1]$ where
$\mathcal{R}_2^*(\alpha)=\cap_{\lambda \in [0,1] }
\mathcal{R}_2(\alpha,\lambda)$. In fact, $v_2$ has an explicit formula
given in \eqref{eq:uFormula} below and $v_3$ a semi-explicit formula given in
\eqref{eq:stochrep_v3} also below.
\end{proposition}
\begin{proof}[Proof of Proposition~\ref{prop:mainPoisson}]
We begin with $v_3$. The preceding discussion all but gives the existence
proof. In particular, it shows that if $g$ is defined by \eqref{g} and
$\hat \delta$ by \eqref{deltaHat}
then the map
\begin{align*}
  (x,y) \mapsto
  \E_{(x,y)}\bigl[\hat{X}_{\hat{\tau}}^{\hat{\delta}}\bigr] =
  x^{\hat{\delta}} \E_{(x,y)}\bigl[e^{\hat{\delta}\tau}\bigr]= x^{\hat{\delta}} g(\sqrt{x} y)
\end{align*}
is well-defined, positive, and $C^2$ for $\delta  \in (0,\frac25)$ and $(x,y) \in \mathcal{R}_3(\alpha)$.
  Returning to \eqref{eq:vStopping1}, classical results (see, for example, \cite{bass98}) allow us to justify
  the stochastic representation formula for $v_3$, which now can be rewritten
  as
\begin{align}\label{eq:stochrep_v3}
  v_3(x,y) =x^{\hat{\delta}}\Bigl[\bigl(\tfrac{c_1}{\hat{\delta}} +
      c_2\bigr)\E_{(x,y)}\bigl[e^{\hat{\delta}\tau}\bigr] -
    \tfrac{c_1}{\hat{\delta}}\Bigr]=x^{\hat{\delta}}\Bigl[\bigl(\tfrac{c_1}{\hat{\delta}} +
      c_2\bigr)g(\sqrt{x} y)-
    \tfrac{c_1}{\hat{\delta}}\Bigr]\,.
\end{align}
As a consequence of this formula, to prove the scaling it suffices to show that
\begin{align*}
  \E_{(x,y)}\bigl[e^{\hat{\delta}\tau}\bigr]\simI
 \ell^0\,.
\end{align*}
  This is clear, since the only dependence of $\tau$ upon $x$
  and $y$ results from $Z_0$, and $Z_0=x^{\frac12} y= (\ell x)^{\frac12}
  (y\ell^{-\frac12})$ is invariant under $S_\ell^{(1)}$.
We now turn to $v_2$. While in light of the scaling and
  continuity properties of $\T_\lambda$, $h$, and $f$ this result could
  be obtained by abstract means, we employ the method of characteristics
  to produce an explicit solution. Namely,
\begin{align}\label{eq:uFormula}
  v_2(x,y,\lambda)=\Bigl(\frac{x^2+\lambda
      y^2}{|y|}\Bigr)^{\delta}\Bigl[\frac{x}{|y|} +
    \alpha\sqrt{\lambda} +
    \sqrt{\lambda}\Bigl(\frac{1}{\alpha^2+1}\Bigr)^{\frac{\delta}{2}}\Bigr]\,.
\end{align}
The scaling properties, regularity, and positivity follow by direct calculation with \eqref{eq:uFormula}.
\end{proof}
\begin{remark}
  As $\lambda \rightarrow 0$, $v_2(x,y,\lambda)$ given in \eqref{eq:uFormula}
  converges to ${x^{2\delta+1}}{|y|^{-(\delta+1)}}$. This was expected
  since formally taking $\lambda \rightarrow 0$ in \eqref{eq:poissonU}
  produces the equation
\begin{equation}
 \label{eq:poissonU0}
 \left\{
   \begin{aligned}
     (x^2 \partial_x + 2 x y \partial_y)v &= -\frac{x^{2(\delta+1)}}{|y|^{\delta+1}}\quad \text{ on $\mathcal{R}_2(\alpha,0)$}\\
     v(x,y) &= 0\qquad \text{ on } \partial
     B_1(\alpha \sqrt{\lambda})\,.
   \end{aligned}
 \right.
\end{equation}
The solution to this simplified equation is easily seen to be
${x^{2\delta+1}}{|y|^{-(\delta+1)}}$. Even in a setting where
\eqref{eq:poissonU} cannot be solved explicitly, this
simplified equation may well be easier to solve. We will see in
Section~\ref{sec:patch} that the most delicate parts of the patching
require precise information about the limit of $v_2$ as $\lambda \rightarrow 0$. This suggests that the analysis may be feasible even when
\eqref{eq:poissonU} is not solvable.
\end{remark}
\begin{remark}
  For both $v_2$ and $v_3$ we have used specifics of the solutions to verify the
scaling. It is possible to derive the results just from the scaling
of the operators, right-hand sides, and boundary conditions. The
positivity for both solutions also follows from the positivity of the
boundary data and the negative definiteness of the right-hand sides.
\end{remark}

\subsection{Proof of Local Super Lyapunov Property}
Letting $B_R(z)=\{(x,y) \in \RR^2 : |(x,y)-z| \leq R\}$, we state the
following proposition which establishes that $v_2$ and $v_3$ are local super Lyapunov functions.
\label{sec:AssembleLyapunov}
\begin{proposition}\label{prop:superu} Fix any $\delta \in
  (0,\frac25)$ then
  for all $\alpha>0$ sufficiently large, there exist constants
  $m_i>0$ and $R_i>0$ so that if
  $(x,y)\in\mathcal{R}_i(\alpha)$ with $|(x,y)|\geq R_i$, then $v_i$
  satisfies
\begin{align*}
(L v_i)(x,y) &\leq - m_i v_i^{\gamma_i}(x,y)
\end{align*}
for $i=2,3$ where $\gamma_2=\gamma_3=\frac{5 \delta +5}{5\delta +3}$. In addition,
we have the following refined estimate in the second region which
emphasizes its transitional nature and which will
be of later use.  Defining
\begin{align}\label{R2subDef}
    \mathcal{R}_2^{(1)}(\alpha)=\overline{ \mathcal{R}_2(\alpha) \cap
   \mathcal{R}_2^{(2)}(\alpha)^c}\quad\text{and}\quad
 \mathcal{R}_2^{(2)}(\alpha) =\overline{\mathcal{R}_2(\alpha)\cap
    \regionP^+_{-1}(\alpha,1)^c}\, ,
\end{align}
if  $j=1,2$ and  $(x,y) \in  \mathcal{R}_2^{(j)}$, we have
 \begin{align*}
(L v_2)(x,y) &\leq - m_2 v_i^{\gamma_2^{(j)}}(x,y)
\end{align*}
where $\gamma_2^{(1)}=\gamma_3=\frac{5 \delta +5}{5\delta +3}$ and $\gamma_2^{(2)}=\gamma_1=\frac{\delta+1}{\delta}$.
\end{proposition}

\begin{proof}[Proof of Proposition \ref{prop:superu}]
We begin with $v_3$ since it is more straightforward. Observe that
for any $\gamma>1$, using equation \eqref{eq:innerPoisson} and the positivity of $v_3$, one has
\begin{align*}
(L v_3)(x,y) &= (\A v_3)(x,y) - y^2\partial_{x}v_3(x,y) +
    \sigma_x\partial_{xx}v_3(x,y)\\
&\leq -v_3^\gamma(x,y) \left[\frac{c_1 x^{\hat{\delta}+1} -
      y^2|\partial_{x}v_3(x,y)| -
      \sigma_x|\partial_{xx}v_3(x,y)|}{v_3^\gamma(x,y)}\right]\\
&\leq - mv_3^\gamma(x,y)
  \end{align*}
where we define
\begin{align*}
   m= \inf_{(x,y) \in \mathcal{R}_3(\alpha)\cap B_{R}^c(0)}\left[\frac{ c_1 x^{\hat{\delta}+1} -
      y^2|\partial_{x}v_3(x,y)| -
      \sigma_x|\partial_{xx}v_3(x,y)|}{v_3^\gamma(x,y)}\right]\,.
\end{align*}
If for some choice of $\gamma>1$ and $R>0$, one has $m>0$, then it is
clear that $v_3$ is a local super Lyapunov function. To prove that
such $\gamma$ and $R$ exist, we use the scaling and continuity
properties of $v_3$ which were proven in
Proposition~\ref{prop:mainPoisson}.  Observe that every point
$(x,y)\in\mathcal{R}_3(\alpha)$ can be mapped back to a point
$(2\alpha,b)$, where $(x,y)=S_\ell^{(1)}(2\alpha,b)$,
$\ell=\frac{x}{2\alpha}$, and $b=\sqrt{\ell}y \in [-1, 1]$.  Therefore
$v_3$ satisfies the following scaling relations:
  \begin{xalignat*}{2}
   v_3(x,y)&= \ell^{\hat{\delta}} v_3(2\alpha,b) & (\partial_{x}
   v_3)(x,y)&= \ell^{\hat{\delta}-1} (\partial_{x} v_3)(2\alpha,b) \\
 x^{\hat{\delta}+1}&=\ell^{\hat{\delta}+1}(2\alpha)^{\hat{\delta}+1}
 &(\partial_{xx} v_3)(x,y)&= \ell^{\hat{\delta}-2}
   (\partial_{xx} v_3)(2\alpha,b)\,.
 \end{xalignat*}
 These scaling relations lead us to choose
 $\gamma=\frac{5\delta+5}{5\delta+3}$, which is the ratio of the
 exponents of $\ell$ in $\A v_3$ and $v_3$. With this choice of
 $\gamma$, we obtain
\begin{multline*}
  \frac{ c_1 x^{\hat{\delta}+1} - y^2|\partial_{x}v_3(x,y)| -
    \sigma_x|\partial_{xx}v_3(x,y)|}{v_3^{\gamma}(x,y)}\\
  = \frac{c_1 (2\alpha)^{\hat{\delta}+1} - \ell^{-3} (b^2 |\partial_x
    v_3(2\alpha,b)| +
    \sigma_x|\partial_{xx}v_3(2\alpha,b)|)}{v_3^{\gamma}(2\alpha,b)}\,.
\end{multline*}
Hence if we define $\ell_*=\inf\{ x/(2\alpha) : (x,y) \in
\mathcal{R}_3(\alpha)\cap B_R^c(0)\}$ and
\begin{align*}
  M=\sup_{b \in[-1,1]}  b^2|\partial_x v_3(2\alpha,b)| +
  \sigma_x|\partial_{xx}v_3(2\alpha,b)| + v_3(2\alpha,b)\, ,
\end{align*}
the preceding estimate and the strict positivity of $v_3$ imply that
\begin{align*}
  m \geq \frac{c_1 (2\alpha)^{\hat{\delta}+1} - \ell_*^{-3}
    M}{M^\gamma}\,.
\end{align*}
Since $v_3$ is $C^2$, we know that $M < \infty$ (the
supremum is over a compact set). Furthermore, observe that $\ell_*
\rightarrow \infty $ as $R \rightarrow \infty$. Combining these last
two observations with the above estimate, we can choose $R$ sufficiently
large to ensure that
\begin{align*}
  c_1 (2\alpha)^{\hat{\delta}+1} - \ell_*^{-3}
M>0
\end{align*}
and hence that $m>0$.  We define $R_3$ and $\gamma_3$ to be the
above choices of $R$ and $\gamma$, respectively, which are valid in
$\mathcal{R}_3(\alpha)$. Substituting these values in the expression
for $m$, we obtain $m_3$. This completes the proof of the local super
Lyapunov property for $v_3$.

We now turn to proving the corresponding property for $v_2$. We start
as we did for $v_3$, by noting that for any $\gamma>1$
\begin{align*}
  (L v_2)(x,y) &= -h(x,y)+ \left(\sigma_x\partial_{xx}v_2+
    \sigma_y\partial_{yy}v_2 \right)(x,y)\\
  &\leq - v_2^\gamma(x,y) \left[ \frac{h -\sigma_x|\partial_{xx}v_2|-
      \sigma_y|\partial_{yy}v_2| }{ v_2^\gamma}\right](x,y) \leq -m
  v_2^\gamma(x,y) \,.
\end{align*}
where in this case we define
\begin{align*}
  m= \inf_{\substack{(x,y) \in \mathcal{R}_2(\alpha)\\|(x,y)|> R}
  }\left[\frac{h -\sigma_x|\partial_{xx}v_2|-
      \sigma_y|\partial_{yy}v_2| }{ v_2^\gamma}\right](x,y) \,.
\end{align*}
As before, we need to show that there exist $\gamma>1$ and $R>0$ such that $m >0$.
Since $\mathcal{R}_2(\alpha)$ has two different natural
scalings, we decompose this region and handle each subregion
separately. Recall the definition of $\mathcal{R}_2^{(1)}$ and  $\mathcal{R}_2^{(2)}$ from
\eqref{R2subDef}
and observe that $\mathcal{R}_2^{(1)}(\alpha)$ scales well under $S_\ell^{(1)}$
  and $\mathcal{R}_2^{(2)}(\alpha)$ under $S_\ell^{(2)}$.  We define $m^{(i)}$ as
  \begin{align*}
    m^{(i)}=  \inf_{\substack{(x,y) \in \mathcal{R}_2^{(i)}(\alpha)
        \\|(x,y)| >R }}\left[\frac{h -\sigma_x|\partial_{xx}v_2|-
\sigma_y|\partial_{yy}v_2| }{  v_2^\gamma}\right](x,y) \,.
  \end{align*}
  We begin with $\mathcal{R}_2^{(2)}(\alpha)$ since the analysis in
  this subregion is very similar to the previous analysis for $v_3$.

  First, note that the circle of radius $r= 2(\alpha^2+1)$ centered at
  the origin is contained in $\regionO(\alpha,1,1)$. Hence any point
  in $\mathcal{R}_2^{(2)}(\alpha)$ can be connected by a radial line
  contained in $\mathcal{R}_2^{(2)}(\alpha)$ to the part of this
  circle contained in $\mathcal{R}_2^{(2)}(\alpha)$. It follows from
  this that any $(x,y)\in \mathcal{R}_2^{(2)}(\alpha)$ can be written
  in the form $(x,y)=S^{(2)}_\ell(a,b)$ where $\ell= |(x,y)| r^{-1}$
  and $(a,b)$ is a point on the circle of radius $r$ centered at the
  origin. Therefore,
  \begin{xalignat*}{2}
    v_2(x,y) &= \ell^{\delta} v_2(a,b) &
    (\partial_{xx} v_2)(x,y) &= \ell^{\delta -2} (\partial_{xx}
    v_2)(a,b)\\
    h(x,y) &= \ell^{\delta+1}   h(a,b)  & (\partial_{yy} v_2)(x,y) &= \ell^{\delta -2} (\partial_{yy} v_2)(a,b)\,.
\end{xalignat*}
Again, by analogy to the previous case, these scaling
relations
lead us to choose $\gamma=\frac{\delta+1}{\delta}$, which is the ratio
of the exponents of $\ell$ in $T v_2$ and $v_2$.  With this choice of
$\gamma$, if $(x,y)=S^{(2)}_\ell(a,b)$, we have that
\begin{align*}
  \left[\frac{h -\sigma_x|\partial_{xx}v_2|-
\sigma_y|\partial_{yy}v_2|   }{  v_2^\gamma}\right] (x,y)=  \left[\frac{h -\ell^{-3}\sigma_x|\partial_{xx}v_2|-
\ell^{-3}\sigma_y|\partial_{yy}v_2| }{  v_2^\gamma}\right] (a,b)\,.
\end{align*}

Setting $\tilde{\mathcal{R}}_2^{(2)}=\{ (a,b) \in \mathcal{R}_2^{(2)}(\alpha) : |(a,b)|=r\}$, we
define
\begin{align*}
  \rho= \inf_{(a,b) \in \tilde{\mathcal{R}}_2^{(2)}} \frac{h(a,b)}{v_2^{\gamma}(a,b)}\quad\text{and}\quad
  M=\sup_{(a,b) \in \tilde{\mathcal{R}}_2^{(2)}} \left[
    \frac{\sigma_x|\partial_{xx}v_2| + \sigma_y|\partial_{yy}v_2|}{v_2^\gamma}\right] (a,b)\,.
\end{align*}
Letting $\ell_*=\inf\{ |(x,y)|/r : (x,y) \in
\mathcal{R}_2^{(2)}(\alpha) \cap B_R^c(0) \}= R/r$, we observe that
$m^{(2)} \geq \rho -\ell_*^{-3} M$. Because $h$ and $v_2$ are $C^2$ in
$(x,y)$ and strictly positive and $\tilde{\mathcal{R}}_2^{(2)}$ is a
compact set, we conclude that $\rho>0$ and $M<\infty$. Hence one can
choose $R$ sufficiently large in order to ensure that $m^{(2)}\geq\rho
-\ell_*^{-3} M >0$.  Again, we denote these specific choices of $R$
and $\gamma$, which are valid in $\mathcal{R}_2^{(2)}(\alpha)$, by
$R_2^{(2)}$ and $\gamma_2^{(2)}$. Substituting these values in the
expression for $m^{(2)}$, we obtain $m_2^{(2)}$. This completes the
proof that $v_2$ is a local super Lyapunov property in the subregion
$\mathcal{R}_2^{(2)}(\alpha)$.

We now turn to region $\mathcal{R}_2^{(1)}(\alpha)$.
Every point $(x,y) \in \mathcal{R}_2^{(1)}(\alpha)$ can be mapped back to a
point $(a,b)$ on the curve $\{\alpha b=a\}$ such that $(x,y)=S_\ell^{(1)}(a,b)$,
where $\ell = \left(\frac{x}{\alpha y}\right)^{\frac23}$, $a=\alpha^{\frac23}(xy^2)^{\frac13}$, and
$b=\alpha^{-\frac13}(xy^2)^{\frac13}$. Hence we get the scaling relations
\begin{xalignat*}{2}
  v_2(x,y,1)&= \ell^{\hat{\delta}} v_2(a,b, \ell^{-3}) &
  (\partial_{xx}
  v_2)(x,y,1)&= \ell^{\hat{\delta}-2} (\partial_{xx} v_2)(a,b, \ell^{-3}) \\
  h(x,y,1)&=\ell^{\hat{\delta}+1} h(a,b,\ell^{-3}) &(\partial_{yy}
  v_2)(x,y,1)&= \ell^{\hat{\delta}+1} (\partial_{yy}v_2)(a,b,
  \ell^{-3})\,.
\end{xalignat*}
Now using the scaling map $S_\ell^{(1)}$ to map $(a,b) \mapsto
(\alpha,1)$ we obtain
\begin{xalignat*}{2}
 v_2(a,b, \ell^{-3})&= b^\delta v_2(\alpha,1, \ell^{-3})&
(\partial_{xx} v_2)(a,b, \ell^{-3})&=b^{\delta-2}
(\partial_{xx} v_2)(\alpha,1, \ell^{-3})\\
h(a,b,\ell^{-3})&=b^{\delta+1}h(\alpha,1,\ell^{-3}) & (\partial_{yy}
v_2)(a,b, \ell^{-3}) &=b^{\delta-2}  (\partial_{yy}
v_2)(\alpha,1, \ell^{-3})\,.
\end{xalignat*}
Again, these scaling relations in $\mathcal{R}_2^{(1)}(\alpha)$ lead
us to choose $\gamma=\frac{5\delta +5}{5\delta+3}$, which is the ratio
of the exponents of $\ell$ in $T v_2$ and $v_2$. Combining these two
sets of scaling estimates and setting $\bar
\gamma=\delta(1-\gamma)+1$, we get that
\begin{align*}
  \Bigl[&\frac{h -\sigma_x|\partial_{xx}v_2|-
      \sigma_y|\partial_{yy}v_2|}{v_2^\gamma}\Bigr] (x,y,1)\\
  &\qquad =b^{\bar \gamma}\left[\frac{h
      -(b\ell)^{-3}\sigma_x|\partial_{xx}v_2|-
      b^{-3}\sigma_y|\partial_{yy}v_2|   }{  v_2^\gamma}\right](\alpha,1,\ell^{-3})\\
  &\qquad =b^{\bar \gamma}\Big[\frac{h}{v_2^\gamma}\Big( 1 -
  (b\ell)^{-3}\sigma_x\frac{|\partial_{xx}v_2|}{h}-
  b^{-3}\sigma_y\frac{|\partial_{yy}v_2|}{h}
  \Big)\Big](\alpha,1,\ell^{-3}) \,.
\end{align*}We have organized this calculation a bit differently for
reasons which will become clear momentarily.
Based on the preceding calculation, we define
\begin{align*}
  \rho(\lambda)=\inf_{\lambda' \in [0,\lambda]}
  \frac{h(\alpha,1,\lambda')}{v_2^\gamma(\alpha,1,\lambda')}
  &\quad\text{and}\quad M_1(\lambda)=\inf_{\lambda'\in[0,\lambda]}
 \frac{|\partial_{xx}v_2 (\alpha,1,\lambda')|}{h (\alpha,1,\lambda')}\\
 &\quad\text{and}\quad M_2(\lambda)=\inf_{\lambda'\in[0,\lambda]}
\frac{  |\partial_{yy}v_2(\alpha,1,\lambda')|}{h(\alpha,1,\lambda')}\,.
\end{align*}
Notice that, unlike the previous calculations, we have made the
constants $\rho$, $M_1$, and $M_2$ depend on the maximal value of
$\lambda$. This is because in our current setting we require more precise information about
these constants than merely that they are finite and positive.

We set $\ell_*=\inf\{ (x/\alpha y)^{\frac23} : (x,y) \in
\mathcal{R}_2^{(1)}(\alpha)\cap B_R^c(0)\}$ and we observe that since $b\geq1$,
\begin{align*}
  m^{(1)} \geq \rho(\ell_*^{-3}) (1- \ell_*^{-3} \sigma_x
  M_1(\ell_*^{-3} ) - \sigma_yM_2(\ell_*^{-3}))\,.
\end{align*}
We wish to conclude that the right-hand side of the above expression
is strictly positive. To conclude this, however, we need to understand
the behavior of $M_1(\lambda)$ and $M_2(\lambda)$ as $\lambda \to
0$. By direct calculation from the explicit formula for $v_2$, we see
that
\begin{equation*}
  M_1(0)=\frac{2\delta(2\delta+1)}{\alpha^3}  \quad \textrm{and} \quad M_2(0)=\frac{(\delta+1)(\delta+2)}{\alpha}\,.
\end{equation*}
Since $M_1(\lambda)$ and $M_2(\lambda)$ are both continuous functions
of $\lambda$ on $(0, 1]$ with finite limits as $\lambda \to 0$, and
since $\ell_*$ can be made arbitrarily large by choosing $R$
sufficiently large, for any $\epsilon>0$ we can choose $R$ large
enough to ensure
\begin{align*}
  1- \ell_*^{-3} \sigma_x M_1(\ell_*^{-3} ) -
  \sigma_yM_2(\ell_*^{-3}) \geq 1- \frac{\sigma_y
    (\delta+1)(\delta+2) }{\alpha} -\epsilon\,.
\end{align*}
We conclude that as long as $ \frac{\sigma_y (\delta+1)(\delta+2)
}{\alpha}<1$, we can always choose $R$ large enough to guarantee that
$m^{(1)}$ is positive. This last inequality holds whenever $\alpha$ is
sufficiently large. Again, we denote these specific choices of $R$ and
$\gamma$ , which are valid in $\mathcal{R}_2^{(1)}(\alpha)$, by
$R_2^{(1)}$ and $\gamma_2^{(1)}$.  Substituting these values in the
expression for $m^{(1)}$, we obtain $m_2^{(1)}$.  Choosing
$m_2=\min\{m_2^{(1)},m_2^{(2)}\}$,
$\gamma_2=\min\{\gamma_2^{(1)},\gamma_2^{(2)}\}$, and $R_2= \max\{R_2^{(1)},R_2^{(2)}\}$
completes the proof that $v_2$ is a local super Lyapunov function in
the entire region $\mathcal{R}_2(\alpha)$. The more detailed
statements of the behavior in $\mathcal{R}_2^{(1)}$ and
$\mathcal{R}_2^{(2)}$ merely serve to summarize the above points.
\end{proof}

\section{Construction of a Global Super Lyapunov Function}
\label{sec:patch}
We now patch together the three local Lyapunov functions that are
defined in distinct regions of the plane in order to produce one
smooth, global Lyapunov function defined on the entire plane.  To do
this, we use the standard mollifier $\phi(t)$, a smooth, increasing
function which varies from 0 to 1 and is suitably normalized to
integrate to unity on the whole real line. Specifically, we take
$\phi(t) =\frac{1}{m}\int_{-\infty}^t \psi(s)ds$ with $m=\int_{-\infty}^\infty \psi(s)ds$
where
\begin{align*}
\psi(t) &=\begin{cases}
\exp\Bigl(\frac{-1}{1-(2t-1)^2}\Bigr) & \text{ for } 0<t<1\\
0 & \text{ otherwise }\,.
\end{cases}
\end{align*}
Next, we define the functions $h_1(x,y)$ and $h_2(x,y)$ as follows:
\begin{equation*}
h_1(x,y)=2+\frac{\alpha|y|}{x}\qquad h_2(x,y)=2-\frac{xy^2}{\alpha}\,.
\end{equation*}
The function $h_1(x,y)=0$ on one boundary of the wedge-shaped region
$\mathcal{R}_1(\alpha) \cap \mathcal{R}_2 (\alpha)$; $h_1(x,y)=1$ on
the other boundary of this region; and $h_1$ varies smoothly between
$0$ and $1$ in the interior.  Similarly, $h_2(x,y)=0$ on one boundary
of the funnel-like region $\mathcal{R}_2(\alpha) \cap
\mathcal{R}_3(\alpha)$ and $h_2(x,y)=1$ on the other boundary.  Thus,
outside of a fixed ball, we define our global Lyapunov function $V$ to
agree with the local Lyapunov functions in subregions of their domains
of definition and to be a smooth, convex combination of the two local
Lyapunov functions in regions of intersection.  In particular, let
$\tilde{V}(x,y)$ be given by
\begin{align*}
  \tilde{V}(x,y)=\begin{cases}
    v_1(x,y) & \text{ for } (x,y)\in\mathcal{R}_1(\alpha)\cap\mathcal{R}_2(\alpha)^c\\
    V_1(x,y) & \text{ for } (x,y)\in\mathcal{R}_1(\alpha)\cap\mathcal{R}_2(\alpha)\\
    v_2(x,y) & \text{ for } (x,y)\in\mathcal{R}_2(\alpha)\cap\mathcal{R}_1(\alpha)^c\cap\mathcal{R}_3(\alpha)^c\\
    V_2(x,y) & \text{ for } (x,y)\in\mathcal{R}_2(\alpha)\cap\mathcal{R}_3(\alpha)\\
    v_3(x,y) & \text{ for } (x,y)\in\mathcal{R}_3(\alpha)\cap\mathcal{R}_2(\alpha)^c\\
    0 & \text{ otherwise }
  \end{cases}
\end{align*}
where
\begin{align*}
V_1(x,y) &= [1-\phi(h_1(x,y))]v_2(x,y)+\phi(h_1(x,y))v_1(x,y)\\
V_2(x,y) &= [1-\phi(h_2(x,y))]v_2(x,y)+\phi(h_2(x,y))v_3(x,y)\,.
\end{align*}
We then choose $V(x,y)\in C^2(\RR^2)$ to satisfy
\begin{align*}
  V(x,y)=
  \begin{cases}
    \tilde{V}(x,y) & \text{ for } x^2+y^2>\rho^2\\
    \text{arbitrary positive and smooth} & \text{ for } x^2+y^2\leq\rho^2
  \end{cases}
\end{align*}
where $\rho$ will be specified below.
\begin{remark}
 At the start of the Lyapunov construction in
 Section~\ref{sec:priming}, we fix a choice of $\delta>0$ when
 defining $v_1$.  This choice is then propagated through our construction
 and is explicitly present in the definition of $v_2$ and $v_3$. During the
 analysis of $v_3$, we note in
 Proposition~\ref{prop:moments} and Proposition~\ref{prop:mainPoisson}
that we must choose $\delta \in (0,\frac25)$.  Except for
this one restriction, we are free to choose $\delta$. Hence our
construction of $V$ depends on two parameters $\delta$ and $\rho$.  As
we will summarize in Proposition~\ref{prop:Vproperties} below, $\delta$ gives the power
of the polynomial growth in all but the pure, positive
$x$-direction. On the other hand, $\rho$ sets the distance from the origin past which
the Lyapunov estimates are valid.
\end{remark}

Consolidating our results on the scaling behavior of the functions $v_i$, $i=1,2,3$, we
obtain the following:
\begin{proposition}\label{prop:Vproperties} Fixing a $\delta \in
  (0,\frac25)$, there exists positive constants  $c$ and $C$, so that $c
  |(x,y)|^\delta \leq V(x,y) \leq C|(x,y)|^{\hat \delta}=C|(x,y)|^{\frac52\delta+\frac32}$.
\end{proposition}
\begin{proof}[Proof of Proposition~\ref{prop:Vproperties}]
  After observing that $\delta < \hat \delta=\frac52\delta+\frac32$ on
  $(0,\frac25)$, the result follows quickly from
  Proposition~\ref{prop:mainPoisson} and the definition of $v_1$ from
  Section~\ref{sec:priming}. On the right half-plane, the result follows from the
definition of $v_1$. On the left half-plane, one can either use the
the scaling relations or the explicit representations given in
\eqref{eq:uFormula} and \eqref{eq:stochrep_v3} to obtain the desired inequalities.
\end{proof}

The following proposition states that $V$ is a super Lyapunov function.  Therefore, one of our main theorems, Theorem~\ref{thm:existsLyap} from
Section~\ref{sec:deterministic-model}, is an immediate consequence of this proposition.
\begin{proposition}\label{prop:superpatch}  For any $\delta \in
  (0,\frac25)$, there exists a $\rho$ from
  the definition of $V$ so that
$V(x,y)$ is a global super Lyapunov function on $\RR^2$.
\end{proposition}
In light of Proposition~\ref{prop:superv1}  and
Proposition~\ref{prop:superu}, the main missing ingredient in the
proof of Proposition~\ref{prop:superpatch} is the verification that $V$ is a local Lyapunov
function in the patching regions. This is the content of the next
proposition; we prove it before returning to the proof of
Proposition~\ref{prop:superpatch} at
the end of the section.

\begin{lemma}\label{lemma:patchonetwo} For any  $\delta \in
  (0,\frac25)$,
$V_1(x,y)$ is a local super Lyapunov function on
$\mathcal{R}_1(\alpha)\cap\mathcal{R}_2(\alpha)$ and  $V_2(x,y)$ is a
local super Lyapunov function on
 $\mathcal{R}_2(\alpha)\cap\mathcal{R}_3(\alpha)$.
\end{lemma}
\begin{proof}[Proof of Lemma~\ref{lemma:patchonetwo}]
 Let $m_i$, $R_i$ and $\gamma_i$ for $i=1,2,3$ be the constants from Proposition~\ref{prop:superv1}  and
 Proposition~\ref{prop:superu}. Next define $m_*=\min\{m_1,m_2,m_3\}$,
 $R^*=\max\{R_1,R_2,R_3\}$ and
 $\gamma_*=\min\{\gamma_1,\gamma_2,\gamma_3 \}= \frac{5 \delta + 5}{5
   \delta +3}$. We further increase $R_*$ so that if  $(x,y) \in
 \mathcal{R}_i$ and $|(x,y)| \geq
 R_*$ then $v_i(x,y) > 1$.
 In proving that $V_2$ is a local super Lyapunov function we need to recall the more refined version of the super Lyapunov
 property in $\mathcal{R}_2$ given in Proposition~\ref{prop:superu}.

We address
  each of the claims in the lemma separately.
We begin with the proof that $V_1$ is super Lyapunov since it is the
most straightforward. If $\rho>R^*$, we have that for all $(x,y)\in\mathcal{R}_1(\alpha)\cap\mathcal{R}_2(\alpha)\cap B_{\rho}^c(0)$
\begin{align*}
  (LV_1)(x,y) &= \ (1-\phi(h_1(x,y)))(Lv_2)(x,y) + \phi(h_1(x,y))(Lv_1)(x,y) + E_1(x,y)\\
  &\leq-m_*[(1-\phi(h_1(x,y)))v_2^{\gamma_1}(x,y)+\phi(h_1(x,y))v_1^{\gamma_1}(x,y)]+E_1(x,y)\\
  &\leq-m_*[V_1(x,y)]^{\gamma_1}+E_1(x,y)\text{ by convexity}\\
  &\leq-m_*(1-M_1)[V_1(x,y)]^{\gamma_1}
\end{align*}
where $M_1$ and $E_1(x,y)$ are defined as
\begin{equation*}
  M_1=\sup_{\substack{(x,y)\in\mathcal{R}_1(\alpha)\cap\mathcal{R}_2(\alpha)\\
      |(x,y)|> \rho}}\left[\frac{E_1(x,y)}{m_*[V_1(x,y)]^{\gamma_1}}\right]
\end{equation*}
and
\begin{align*}
  E_1(x,y) =& \ L[\phi(h_1(x,y))](v_1(x,y)-v_2(x,y))\\
  & + 2\sigma_x\frac{\partial}{\partial x}[\phi(h_1(x,y))]\frac{\partial}{\partial x}[v_1(x,y)-v_2(x,y)]\\
  & + 2\sigma_y\frac{\partial}{\partial y}[\phi(h_1(x,y))]\frac{\partial}{\partial y}[v_1(x,y)-v_2(x,y)]\,.
\end{align*}
If we can choose $\rho$ sufficiently large so that $M_1<1$, then
$V_1(x,y)$ will be a local super Lyapunov function on
$\mathcal{R}_1(\alpha)\cap \mathcal{R}_2(\alpha)$. To show $M_1<1$, we
use the scaling properties of $v_1$ and $v_2$ to map back to a
circular arc of radius $r=\sqrt{\alpha^2+4}$.  Let
\begin{align*}
  \ell=\frac{\sqrt{x^2+y^2}}{r}, \qquad a=\frac{x}{\ell}, \qquad\text{and}\quad b=\frac{y}{\ell}\,.
\end{align*}
Then
\begin{align*}
  (x,y)\in\mathcal{R}_1(\alpha)\cap\mathcal{R}_2(\alpha)\cap
  B_{\rho}^c(0) \quad \Longrightarrow \quad (x,y)=S_\ell^{(2)}(a,b) \quad\text{with}\quad
  \ell\geq1
\end{align*}
Note that $h_1(x,y)=h_1(a,b)$, so
\begin{align*}
  V_1(x,y)=\ell^{\delta}V_1(a,b) \quad\text{ and }
  \quad
  [V_1(x,y)]^{\gamma_1}=\ell^{\delta+1}[V_1(a,b)]^{\gamma_1}\,.
\end{align*}
As a consequence of these scaling relations, we get that for all $(x,y)\in\mathcal{R}_1(\alpha) \cap \mathcal{R}_2(\alpha)\cap B_{\rho}^c(0)$,
\begin{align*}
  E_1(x,y) 
  =& \ \ell^{\delta+1}\alpha\phi'(h_1(a,b))|b|\left(1+\frac{b^2}{a^2}\right)[v_1(a,b)-v_2(a,b)]\\
  & + \ell^{\delta-1}\alpha\phi''(h_1(a,b))\left(\frac{-\sigma_x|b|}{a^2}+\frac{\sigma_y \sgn(b)}{a}\right)[v_1(a,b)-v_2(a,b)]\\
  & + \ell^{\delta-2}\alpha\phi'(h_1(a,b))\frac{2\sigma_x|b|}{a^3}[v_1(a,b)-v_2(a,b)]\\
  & + \ell^{\delta-2}\alpha\phi'(h_1(a,b))\frac{-2\sigma_x|b|}{a^2}\frac{\partial}{\partial x}[v_1(a,b)-v_2(a,b)])\\
  & + \ell^{\delta-2}\alpha\phi'(h_1(a,b))\frac{2\sigma_y
    \sgn(b)}{a}\frac{\partial}{\partial y}[v_1(a,b)-v_2(a,b)]\,.
\end{align*}
Hence we have that
\begin{equation}\label{eq:M1_sup}
  M_1\leq\sup_{\substack{(a,b)\in\mathcal{R}_1(\alpha)\cap\mathcal{R}_2(\alpha)\\|(x,y)|
      \leq r}}\left[\frac{e_{1,1}(a,b)}{m_*[V_1(a,b)]^{\gamma_1}}+\frac{e_{1,2}(a,b)}{\ell^2 m_*[V_1(a,b)]^{\gamma_1}}\right]
\end{equation}
where
\begin{align*}
  e_{1,1}(a,b) =& \ \alpha\phi'(h_1(a,b))|b|\Bigl(1+\frac{b^2}{a^2}\Bigr)[v_1(a,b)-v_2(a,b)]\\
  e_{1,2}(a,b) =& \ \alpha\phi''(h_1(a,b))\Bigl(\frac{-\sigma_x|b|}{a^2}+\frac{\sigma_y \sgn(b)}{a}\Bigr)[v_1(a,b)-v_2(a,b)]\\
  & + \alpha\phi'(h_1(a,b))\frac{2\sigma_x|b|}{a^3}[v_1(a,b)-v_2(a,b)]\\
  & + \alpha\phi'(h_1(a,b))\frac{-2\sigma_x|b|}{a^2}\frac{\partial}{\partial x}[v_1(a,b)-v_2(a,b)]\\
  & + \alpha\phi'(h_1(a,b))\frac{2\sigma_y \sgn(b)}{a}\frac{\partial}{\partial y}[v_1(a,b)-v_2(a,b)]\,.
\end{align*}
By explicit computation with $v_1$ and $v_2$, we can show that
$e_{1,1}(a,b)$ is always negative for $(a,b)$ in the desired region.
The second term of the sum in \eqref{eq:M1_sup}, the upper bound for
$M_1$, can then be made arbitrarily small by choosing $\ell$ large
enough; this corresponds to choosing $\rho$ sufficiently large.  This
establishes that $M_1<1$, which completes the proof of the lemma.

We now turn to proving that $V_2$ is super Lyapunov.  If $\rho>R^*$,
then for all
$(x,y)\in\mathcal{R}_2(\alpha)\cap\mathcal{R}_3(\alpha)\cap
B_{\rho}^c(0)$
\begin{align*}
  (LV_2)(x,y) &= \ (1-\phi(h_2(x,y)))(Lv_2)(x,y) + \phi(h_2(x,y))(Lv_3)(x,y) + E_2(x,y)\\
  &\leq-m_*[(1-\phi(h_2(x,y)))v_2^{\gamma_3}(x,y)+\phi(h_2(x,y))v_3^{\gamma_3}(x,y)]+E_2(x,y)\\
  &\leq-m_*[V_2(x,y)]^{\gamma_3}+E_2(x,y)\quad\text{ by convexity}\\
  &\leq-m_*(1-M_2)[V_2(x,y)]^{\gamma_3}
\end{align*}
where
\begin{align*}
  M_2=\sup_{(x,y)\in\mathcal{R}_2(\alpha)\cap\mathcal{R}_3(\alpha)\cap
    B_{\rho}^c(0)}\left[\frac{E_2(x,y)}{m_*[V_2(x,y)]^{\gamma_3}}\right]
\end{align*}
and
\begin{align*}
  E_2(x,y) =& \ L[\phi(h_2(x,y))](v_3(x,y)-v_2(x,y))\\
  & + 2\sigma_x\frac{\partial}{\partial x}[\phi(h_2(x,y))]\frac{\partial}{\partial x}[v_3(x,y)-v_2(x,y)]\\
  & + 2\sigma_y\frac{\partial}{\partial
    y}[\phi(h_2(x,y))]\frac{\partial}{\partial y}[v_3(x,y)-v_2(x,y)]\,.
\end{align*}
If $M_2<1$, then $V_2(x,y)$ will be a super Lyapunov function on
$\mathcal{R}_2(\alpha)\cap\mathcal{R}_3(\alpha)$. To show $M_2<1$, we
use the scaling properties of $v_2$ and $v_3$ to map back to a
vertical line.  Let
\begin{align*}
  \ell=\frac{x}{2\alpha}, \qquad a=2\alpha,\qquad \text{and}\qquad b=y\sqrt{\ell}\,.
\end{align*}
Then
\begin{equation*}
  (x,y)\in\mathcal{R}_2(\alpha)\cap\mathcal{R}_3(\alpha)\cap
  B_{\rho}^c(0) \quad \Longrightarrow \quad (x,y,1)=S_\ell^{(1)}(a,b,\ell^{-3})
\end{equation*}
where $|b|\in\bigl[\frac{1}{\sqrt{2}},1\bigr]$ and $\ell\geq1$.
Note that $h_2(x,y)=h_2(a,b)$, so $V_2(x,y)=V_2(x,y,1)$ satisfies
\begin{equation*}
  V_2(x,y,1)=\ell^{\hat{\delta}}V_2(a,b,\ell^{-3}) \quad\text{and}\quad
  [V_2(x,y,1)]^{\gamma_3}=\ell^{\hat{\delta}+1}[V_2(a,b,\ell^{-3})]^{\gamma_3}
\end{equation*}
where
\begin{equation*}
  V_2(x,y,\lambda)=[1-\phi(h_2(x,y))]v_2(x,y,\lambda)+\phi(h_2(x,y))v_3(x,y)\,.
\end{equation*}
Now we have that for all
$(x,y)\in\mathcal{R}_2(\alpha)\cap\mathcal{R}_3(\alpha)\cap
B_{\rho}^c(0)$,
\begin{align*}
  E_2(x,y) 
  =& \ \ell^{\hat{\delta}+1}\frac{\phi'(h_2(a,b))}{\alpha}(-5a^2b^2-2a\sigma_y)[v_3(a,b)-v_2(a,b,\ell^{-3})]\\
  & + \ell^{\hat{\delta}+1}\frac{\phi'(h_2(a,b))}{\alpha}(-4ab\sigma_y)\frac{\partial}{\partial y}[v_3(a,b)-v_2(a,b,\ell^{-3})]\\
  & + \ell^{\hat{\delta}+\frac12}\frac{\phi''(h_2(a,b))}{\alpha}(-2ab\sigma_y)[v_3(a,b)-v_2(a,b,\ell^{-3})]\\
  & + \ell^{\hat{\delta}-1}\frac{\phi''(h_2(a,b))}{\alpha}(-b^2\sigma_x)[v_3(a,b)-v_2(a,b,\ell^{-3})]\\
  & + \ell^{\hat{\delta}-2}\frac{\phi'(h_2(a,b))}{\alpha}b^4[v_3(a,b)-v_2(a,b,\ell^{-3})]\\
  & + \ell^{\hat{\delta}-2}\frac{\phi'(h_2(a,b))}{\alpha}(-2b^2\sigma_x)\frac{\partial}{\partial x}[v_3(a,b)-v_2(a,b,\ell^{-3})]\,.
\end{align*}
Define $N(\lambda^*)$ as follows:
\begin{equation}\label{eq:N(lambda*)}
  N(\lambda^*)=\sup_{\substack{|b|\in[\frac{1}{\sqrt{2}},1]\\\lambda\in(0,\lambda^*]}}\left[\frac{e_{2,1}(a,b,\lambda)}{m_*[V_2(a,b,\lambda)]^{\gamma_3}}+\frac{e_{2,2}(a,b,\lambda)}{\sqrt{\ell}m_*[V_2(a,b,\lambda)]^{\gamma_3}}\right]
\end{equation}
where
\begin{align*}
  e_{2,1}(a,b,\lambda) =& \ \frac{\phi'(h_2(a,b))}{\alpha}(-5a^2b^2-2a\sigma_y)[v_3(a,b)-v_2(a,b,\lambda)]\\
  & + \frac{\phi'(h_2(a,b))}{\alpha}(-4ab\sigma_y)\frac{\partial}{\partial y}[v_3(a,b)-v_2(a,b,\lambda)]\,,\\
  e_{2,2}(a,b,\lambda) =& \ \frac{\phi''(h_2(a,b))}{\alpha}(-2ab\sigma_y)[v_3(a,b)-v_2(a,b,\lambda)]\\
  & + \frac{\phi''(h_2(a,b))}{\alpha}(-b^2\sigma_x)[v_3(a,b)-v_2(a,b,\lambda)]\\
  & + \frac{\phi'(h_2(a,b))}{\alpha}b^4[v_3(a,b)-v_2(a,b,\lambda)]\\
  & + \frac{\phi'(h_2(a,b))}{\alpha}(-2b^2\sigma_x)\frac{\partial}{\partial x}[v_3(a,b)-v_2(a,b,\lambda)]\,.
\end{align*}
Note that for any $\lambda^* >0$, we can choose $\rho$ sufficiently large to force $M_2$ (which, we recall, depends on $\rho$) to be less than $N(\lambda^*)$.  Ultimately, we will choose $\lambda^*$ sufficiently small so that $N(\lambda^*)<1$.
The second term of the sum in \eqref{eq:N(lambda*)} can be made
arbitrarily small by increasing the size of $\ell$; again, increasing
the size of $\ell$ corresponds to increasing $\rho$.  We now address
the first term of the sum in \eqref{eq:N(lambda*)}. From
Lemma~\ref{claim:M2} which is stated and proven bellow, we see that we
can chose the parameters to make this term negative.

Combining all of these observations, we have demonstrated that $M_2<1$, which completes the
proof of the lemma.
\end{proof}

\begin{lemma}\label{claim:M2} There exist positive constants $c_1$ and $c_2$ in the
  definition of the Poisson equation for $v_3(x,y)$, and positive
  $\alpha$ and $\lambda^*$ such that for all $\lambda \in [0,
  \lambda^*]$, the following inequalities hold for $a=2\alpha$ and
  $|b| \in\bigl[\tfrac{1}{\sqrt{2}}, 1\bigr]$:
  \begin{align}
    v_3(a,b)-v_2(a,b,\lambda) & > 0 \label{eq:diff_v3_v2}\\
    b\Bigl[\frac{\partial v_3}{\partial y}(a,b)-\frac{\partial
        v_2}{\partial
        y}(a,b,\lambda)\Bigr]&>0 \label{eq:deriv_diff_v3_v2}
  \end{align}
\end{lemma}
\begin{proof}[Proof of Lemma~\ref{claim:M2}]
  Recall that, from \eqref{eq:stochrep_v3}, $v_3(x,y)$ can be represented as
  \begin{equation}\label{eq:stochrep2_v3}
    v_3(x,y)=x^{\hat{\delta}}\Bigl[\bigl(\tfrac{c_1}{\hat{\delta}}+c_2\bigr)\E_{(x,y)}\bigl[e^{\hat{\delta}\tau}\bigr]
      - \tfrac{c_1}{\hat{\delta}}\Bigr]
  \end{equation}
  where $\tau=\inf \{t>0: |Z_t| \notin [-\sqrt{2\alpha},
  \sqrt{2\alpha}]\}$ and $Z_t$ is the process given in
  \eqref{eq:Zt_defn}.  Note that the expectation in
  \eqref{eq:stochrep2_v3} can be written as the solution to a
  second-order ODE, namely:
  \begin{equation*}
    \E_{(x,y)}\bigl[e^{\hat{\delta}\tau}\bigr]=g_{\ep}(\sqrt{\ep x}\ y)
  \end{equation*}
  where $g_{\ep}(z)$ solves the following boundary value problem with $\ep=\frac{1}{2\alpha}$:
  \begin{equation}\label{eq:g_ep_defn}
    \begin{cases}
      \epsilon \sigma_y g''_{\ep}(z) + \frac52 z g'_{\ep}(z) + \hat{\delta} g_{\ep} (z)=0 & \text{ for } z \in (-1, 1) \\
      g_{\ep}(-1)=g_{\ep}(1)=1\,. &
    \end{cases}
  \end{equation}
  Define $g_0(z)$ to be the solution to the limiting ODE in \eqref{eq:g_ep_defn} when $\epsilon=0$
  and note that $g_0(z)$ can be computed exactly for initial conditions $z \neq 0$:
  \begin{equation}\label{eq:g_0_soln}
    g_0(z)=\frac{1}{|z|^{\delta + \frac35}}\,.
  \end{equation}
Now, let $v_0(x,y)$ be defined as
\begin{equation}\label{eq:v_0_defn}
  v_3^0(x,y) = x^{\hat{\delta}}\Bigl[\bigl(\tfrac{c_1}{\hat{\delta}}+c_2\bigr)g_0(\sqrt{\ep
      x}\ y) - \tfrac{c_1}{\hat{\delta}}\Bigr]\,.
\end{equation}
We address the first difference between $v_3$ and $v_2$ in \eqref{eq:diff_v3_v2} as follows.
We write
\begin{align}
  v_3(a,b)-v_2(a,b)=&v_3(a,b)-v_3^0(a,b) \label{eq:d1}\\
  &+v_3^0(a,b)-v_2(a,b,0)  \label{eq:d2}\\
  &+v_2(a,b,0)-v_2(a,b,\lambda) \label{eq:d3}\,.
\end{align}
To show that this difference is positive, we will show that
$v_3^0(a,b)-v_2(a,b,0)>0$ and that the other two differences are small
in comparison.  Similarly, for the difference between the
$y$-derivatives of $v_3$ and $v_2$ in \eqref{eq:deriv_diff_v3_v2}, we
write
\begin{align}
  b\Bigl[\frac{\partial v_3}{\partial y}(a,b)-\frac{\partial
      v_2}{\partial y}(a,b,\lambda)\Bigr]=&
  b\Bigl[\frac{\partial v_3}{\partial y}(a,b)-\frac{\partial v_3^0}{\partial y}(a,b)\Bigr] \label{eq:dyd1}\\
  &+b\Bigl[\frac{\partial v_3^0}{\partial y}(a,b) - \frac{\partial v_2}{\partial y}(a,b,0)\Bigr] \label{eq:dyd2} \\
  &+b\Bigl[\frac{\partial v_2}{\partial y}(a,b,0) - \frac{\partial
      v_2}{\partial y}(a,b,\lambda)\Bigr]\label{eq:dyd3}
\end{align}
and again, we will show that $b\bigl[\frac{\partial v_3^0}{\partial
    y}(a,b) - \frac{\partial v_2}{\partial y}(a,b,0)\bigr]>0$ and the
other two differences are small in comparison. Specifically, we
demonstrate that there exist positive constants $c_1$ and $c_2$ in the
Poisson equation for $v_3$ such that the differences in \eqref{eq:d2}
and \eqref{eq:dyd2} are positive; and then, that there exists $\alpha$
sufficiently large such that the differences on the righthand sides of \eqref{eq:d1} and
\eqref{eq:dyd1} are comparatively small; and last, that there exists a
$\lambda^*$ such that \eqref{eq:d3} and \eqref{eq:dyd3} are
comparatively small for all $\lambda \in [0, \lambda^*]$.  For the
first of these claims, note that
\begin{align}
  v_3^0(a,b)-v_2(a,b,0)&=\frac{(2\alpha)^{2\delta +1}}{|b|^{\delta+1}}q(b) \label{eq:v3_v2_q}\\
  b\Bigl[\frac{\partial v_3^0}{\partial y}(a,b) - \frac{\partial
      v_2}{\partial y}(a,b,0)\Bigr]& =
  \frac{(2\alpha)^{2\delta+1}}{|b|^{\delta+1}}\tilde{q}(b) \notag 
\end{align}
where $q$ and $\tilde{q}$ are given by
\begin{align*}
  q(b)&= 2^{\frac12 \delta +
    \frac12}\Bigl[\bigl(\tfrac{\tilde{c_1}}{\hat{\delta}}+\tilde{c_2}\bigr)|b|^{\frac25}-
    \frac{\tilde{c_1}}{\hat{\delta}}|b|^{\delta+1}\Bigr]-1\\
  \tilde{q}(b)&=-\bigl(\delta + \tfrac35\bigr)2^{\frac12 \delta +
    \frac12}\bigl(\tfrac{\tilde{c_1}}{\hat{\delta}}+\tilde{c_2}\bigr)|b|^{\frac25}+\delta+1
\end{align*}
and $c_1=\frac{\tilde{c_1}}{\alpha^{\frac12\delta+\frac12}}$ and
$c_2=\frac{\tilde{c_2}}{\alpha^{\frac12\delta+\frac12}}$.  We note
that $c_1$ and $c_2$ are chosen to scale with $\alpha$ so that $v_2$
and $v_3$ have identical scaling in $\alpha$.  Moreover, as we
demonstrate below, $\tilde{c_1}$ and $\tilde{c_2}$ can be chosen
independently of $\alpha$.  It is clear that $\tilde{q}$ is a monotone
decreasing function of $|b|$; hence it is minimized at the right
endpoint of the interval for $|b|$, that is, $|b|=1$.  Thus if we can
show $\tilde{q}(1)>0$, then it follows that
\begin{equation}\label{eq:pos_deriv_v3_v2}
  b\Bigl[\frac{\partial v_3^0}{\partial y}(a,b) - \frac{\partial v_2}{\partial y}(a,b,0)\Bigr]>0
\end{equation}
for all $|b| \in [2^{-\frac12}, 1]$.  If we can also show
$q(2^{-\frac12})>0$, then from \eqref{eq:v3_v2_q}, we conclude that
\begin{equation*}
v_3^0(a,2^{-\frac12}) - v_2(a, 2^{-\frac12}, 0)>0\,.
\end{equation*}
Combining this with
\eqref{eq:pos_deriv_v3_v2} gives the desired positivity of
\eqref{eq:d2} on the whole interval $|b| \in [2^{-\frac12}, 1]$.
Hence, we need only verify that there exist positive values of
$\tilde{c_1}$ and $\tilde{c_2}$ such that
\begin{align}
  q(2^{-\frac12})&= 2^{\frac12 \delta +
    \frac12}\Bigl[\bigl(\tfrac{\tilde{c_1}}{\hat{\delta}}+\tilde{c_2}\bigr)2^{-\frac15}
    - \tfrac{\tilde{c_1}}{\hat{\delta}}
    2^{-(\frac12 \delta + \frac12)}\Bigr] - 1 > 0 \label{eq:first_eq_for_c1}\\
  \tilde{q}(1)&=-\bigl(\delta + \tfrac35\bigr)2^{\frac12 \delta +
    \frac12}\bigl(\tfrac{\tilde{c_1}}{\hat{\delta}}+\tilde{c_2}\bigr)+\delta+1>0 \label{eq:second_eq_for_c2}\,.
\end{align}
The verification of this is elementary and we omit the details.

We remark that $\tilde{c_1}$ and $\tilde{c_2}$ can be chosen
independently of $\alpha$, since the above inequalities have no
dependence on $\alpha$. Thus, choosing positive $\tilde{c_1}$ and
$\tilde{c_2}$ such that \eqref{eq:first_eq_for_c1} and
\eqref{eq:second_eq_for_c2} are both satisfied, we get that for all
$|b| \in[2^{-\frac12}, 1]$,
\begin{align*}
  v_3^0(a,b)-v_2(a,b,0)&\geq \alpha^{2\delta+1} 2^{\frac52 \delta +\frac32} q(2^{-\frac12})>0 \\
  b\Bigl[\frac{\partial v_3^0}{\partial y}(a,b) - \frac{\partial
      v_2}{\partial y}(a,b,0)\Bigr]&\geq \alpha^{2\delta+1}
  2^{2\delta + 1} \tilde{q}(1) >0\,.
\end{align*}
We turn our attention to making the differences
\begin{equation*}
  v_3(a,b)-v_3^0(a,b) \quad \text{and} \quad b\Bigl[\frac{\partial v_3}{\partial y}(a,b)-\frac{\partial v_3^0}{\partial y} (a,b)\Bigr]
\end{equation*}
comparatively small.  Note that
\begin{align*}
  v_3(a,b)-v_3^0(a,b)&=\alpha^{2\delta +1} 2^{\frac52 \delta + \frac32} \bigl(\tfrac{\tilde{c_1}}{\hat{\delta}}+\tilde{c_2}\bigr) [g_{\ep}(b)-g_0(b)]  \\
  b\Bigl[\frac{\partial v_3}{\partial y}(a,b)-\frac{\partial
    v_3^0}{\partial y} (a,b)\Bigr] &=\alpha^{2\delta +1} 2^{\frac52
    \delta +
    \frac32}b\bigl(\tfrac{\tilde{c_1}}{\hat{\delta}}+\tilde{c_2}\bigr)[g_{\ep}'(b)-g'_0(b)]\,.
\end{align*}
To be precise, we will show that
\begin{align*}
  |v_3(a,b)-v_3^0(a,b)|&<\tfrac13 \bigl[\alpha^{2\delta+1} 2^{\frac52 \delta +\frac32} q(2^{-\frac12})\bigr] \\
  \Bigl|b\Bigl[\frac{\partial v_3}{\partial y}(a,b)-\frac{\partial
        v_3^0}{\partial y} (a,b)\Bigr]\Bigr|&<\tfrac13
  \left[\alpha^{2\delta+1} 2^{2\delta + 1} \tilde{q}(1)\right]\,.
\end{align*}
This is equivalent to establishing that
\begin{align}
  \Bigl|\alpha^{2\delta +1} 2^{\frac52 \delta + \frac32}
    \bigl(\tfrac{\tilde{c_1}}{\hat{\delta}}+\tilde{c_2}\bigr)
    [g_{\ep}(b)-g_0(b)]\Bigr|&<\tfrac13 \bigl[\alpha^{2\delta+1}
    2^{\frac52 \delta +\frac32} q(2^{-\frac12})\bigr]
\label{eq:gdiff}\\
\left|\alpha^{2\delta +1} 2^{\frac52 \delta +
    \frac32}b\bigl(\tfrac{\tilde{c_1}}{\hat{\delta}}+\tilde{c_2}\bigr)[g_{\ep}'(b)-g'_0(b)]\right|&<\tfrac13
\left[\alpha^{2\delta+1} 2^{2\delta + 1}
  \tilde{q}(1)\right]\label{eq:gprime_diff}\,.
\end{align}
Observe that the same powers of $\alpha$ appear on both sides of each
of the above inequalities.  Therefore, to prove \eqref{eq:gdiff} and
\eqref{eq:gprime_diff}, it suffices to show that $g_{\ep}(b)$ and
$g'_{\ep}(b)$ converge uniformly to $g_0(b)$ and $g_0'(b)$,
respectively, for $|b| \in [2^{-\frac12}, 1]$ as
$\epsilon=\frac{1}{2\alpha} \rightarrow 0$.  Both of these uniform
convergences follow from classical results; see, for example,
\cite{abramowitzStegun92}.

Since $\ep=\frac{1}{2\alpha}$, we can choose $\alpha$ sufficiently
large to guarantee that both \eqref{eq:gdiff} and
\eqref{eq:gprime_diff} hold and that $v_2$ remains a local super
Lyapunov function on $\mathcal{R}_2(\alpha)$ (recall that in
Proposition \ref{prop:superu}, a lower bound on the size of $\alpha$
was imposed).  Finally, by choosing $\lambda^*$ sufficiently small,
the differences in \eqref{eq:d3} and \eqref{eq:dyd3} can be made small
for all $\lambda \in [0, \lambda*]$. This is an immediate consequence
of the fact that $v_2(a,b,\lambda)$ is a $C^2$ function of $\lambda
\in [0,1]$.
\end{proof}

Having established the super Lyapunov property in the patching
region, we return to the proof of the main result of this section.

\begin{proof}[Proof of Proposition~\ref{prop:superpatch}]
The local super Lyapunov condition has now been proven in regions; namely,
for $v_1$ in Proposition~\ref{prop:superv1}, for $v_2$ and $v_3$ in
Proposition~\ref{prop:superu}, and for the patched functions $V_1$ and
$V_2$ in Proposition~\ref{lemma:patchonetwo}. All that remains is to make
a global choice of constants. The constant $\rho$ from
Proposition~\ref{lemma:patchonetwo} was chosen to be valid in all
regions. It is sufficient to choose
\begin{align*}
M&=\min\big\{m_*(1-M_1),m_*(1-M_2)\big\} < m_*,\\
b& = \sup\big\{|(LV)(x,y)| : x^2+y^2\leq\rho^2\big\}, \textrm{ and} \\
\gamma&=\min\{\gamma_1,\gamma_2, \gamma_3\}=\frac{5\delta +5}{5\delta+3}.
\end{align*}
These choices guarantee that for all $(x,y)\in\RR^2$
\begin{align*}
 (LV)(x,y)\leq -M \ [V(x,y)]^{\gamma}+b\,.
\end{align*}
\end{proof}

\section{Existence and Positivity of Transition Density}
\label{sec:positivity}
Having established the existence of a global super Lyapunov function, we
now make a small detour to prove the existence of a smooth density
with appropriate positivity properties. These results
provide the missing ingredient to prove the ergodic result
stated in Theorem~\ref{thm:main}, namely a minorization
condition. It is worth noting, however, that proving the minorization condition is not our sole goal. Indeed, if it were, we would not need all of the results of this section: we could simply use smoothness and appropriate open set controllability results. See for
example \cite{MattinglyStuartHigham02,MattinglyMcKinleyPillai09}.

Instead, our interest is larger, motivated by two concerns. First, we
wish to understand the general structure of the invariant measure, not
merely its uniqueness. Second, we want to take this simple example to
highlight some techniques, different than those often used, which can
be applied in more general situations to address questions of positivity. We feel
that these methods convey more intuition and better allow for the
inclusion of \textit{a priori} facts about the dynamics.

\subsection{Positivity when $\sigma_x>0$}
\label{sec:elliptic}
When $\sigma_x>0$ (since we always assume $\sigma_y>0$), the system is uniformly elliptic, and everything is
relatively straightforward. Since the diffusion associated with
\eqref{eq:sde} has a constant, positive definite diffusion matrix,
classical results guarantee the existence of a function $p:
(0,\infty)\times \RR^2 \times\RR^2 \rightarrow (0,\infty)$ such that
$p$ is jointly continuous, $p_t(z,z')$ is strictly positive for all
$(t,z,z')$, and such that for all measurable sets $A$

\begin{align}\label{eq:existsDensity}
  P_t(z,A) = \int_A p_t(z,z')dz'\,.
\end{align}

We summarize these results for future reference in the following Proposition.
\begin{proposition}\label{prop:posDensityEliptic} If $\sigma_y>0$ and $\sigma_x>0$ then for all $t >0$, $P_t$ has an
 everywhere positive  density  $p_t(z,z')$ with respect to Lebesgue
 measure which is  smooth in both $z$ and $z'$.
\end{proposition}

\subsection{Positivity when $\sigma_x=0$}
\label{sec:hypoelliptic}

When $\sigma_x=0$ (but $\sigma_y$ still positive), the situation is more delicate. We begin by
observing that the generator of the associated diffusion is still
hypoelliptic; to see this, we write the generator of the diffusion as
\begin{align*}
  L=X+\frac12 \sigma_y \partial^2_y\,.
\end{align*}
Observe that $[[X,\partial_y],\partial_y]=-2$ and hence the
relevant ideal in the algebra generated by $X$ and $\partial_y$ is of
full rank. In turn, this ensures the  existence of a continuous
function $p$ so that \eqref{eq:existsDensity} holds. The main
difference between this and the setting of Section~\ref{sec:elliptic} is that it
is no longer immediate that $p_t(z,z')$ is positive for all $t>0$ and
$z,z' \in \RR^2$. In fact it is not true.

Intuitively, it is clear that if $z$ is in the left-half plane and
$z'$ in the right-half plane then $p_t(z,z')$ should be zero, since
there is no way to move across the $y$-axis. On the other hand, it is
reasonable to expect that given any $z'$ in the left-half plane, there
exists a $T=T(z')$ such that
$p_t(z,z')>0$ for all $t > T$ and $z \in
\RR^2$.

There are a number of ways to prove such a result. The most
generally applicable and powerful technique is to leverage geometric
control theory to show that the support of $P_t(z,\ccdot)$ contains a
sufficiently large, bounded region of the left-half plane for any $z$
and $t$ sufficiently large. From this, for
example, one can show that $(X_t,Y_t)$ is sufficiently smooth in the
Malliavin sense (which it is), and deduce that $p_t(z,z')$ is
strictly positive in the interior of the support.

Alternatively, one can use sufficiently quantitative open-set
controllability results to extend the very local positivity which
follows just from the joint-continuity of $(z,z') \mapsto
p_t(z,z')$. This is the method employed in
\cite{MattinglyStuartHigham02,MattinglyMcKinleyPillai09} in various forms.

Here we take an approach most consonant with the first
option. However, rather than merely citing the appropriate geometric
control theory results, we construct an explicit series of simple
controls to prove the positivity condition we require. The subsequent
discussion is lengthier, but we feel that it is more intuitive and is
a useful complement to more general control theoretic arguments,
especially for the uninitiated.

Before turning to the existence of a positive density for
\eqref{eq:sde}, we first consider an
analogous calculation in a simpler setting.
Consider a smooth map $\phi\colon \RR^m \rightarrow \RR^d$ where $m >
d$ and let $\Gamma$ be a non-degenerate Gaussian probability measure
on $\RR^m$. Consider the
push forward of $\Gamma$ by $\phi$, denoted by $\Gamma\phi^{-1}$ and
defined by $\Gamma\phi^{-1}(A) = \Gamma( \phi^{-1}(A))$.  For the
measure $\Gamma\phi^{-1}$ to be absolutely continuous with respect to Lebesgue measure, it
is necessary and sufficient that
\begin{align*}
  \Gamma\{ x \in \RR^m : \mathrm{Det}|(D\phi)(x) (D\phi^T)(x)|= 0
  \} =0\,.
\end{align*}
(See \cite{Bell_1995,Bell_2004} for more details.) Supposing that
this condition holds, we let $\hat \gamma$ denote the density of
$\Gamma\phi^{-1}$ with respect to Lebesgue measure. We are interested in when $\hat \gamma(z)$ is
positive at a given point $z \in \RR^d$. It is a simple exercise in
calculus to see that $\hat \gamma(z) >0$ if and only if there exists a
$x \in \RR^m$ with $\gamma(x) = z$ and for which
$(D\phi(x))(D\phi^T(x))$ is a non-degenerate matrix. The first
condition ensures that there is a way to reach $z$; that is to say,
$z$ is the image of $R^d$ under $\phi$. The second ensures that an
infinitesimal piece of volume, and hence probability, is brought with
$x$ when it is mapped by $\phi$.

This intuitive fact has a counterpart in stochastic
analysis. While these ideas rest on the foundation of Malliavin calculus,
the closest analogue is found in the work of Ben Arous and Leandre
\cite{BenArousLeandre91,BenArousLeandre91B} and the subsequent presentation by Nualart\cite{barlowNualart98}. We begin by
identifying the map in question.

To any $U \in L^2([0,T], \RR^d)$, we associate $\{(X_t^U,Y_t^U) : t \in [0,T]\}$ which solves the system of equations
\begin{equation}\label{eq:sde_withcontrol}
\begin{aligned}
\dot X^U_t &= (X^U_t)^2 - (Y^U_t)^2\\
\dot Y^U_t &= 2 X^U_t Y^U_t + U_t
\end{aligned}\,.
\end{equation}
In the following discussion, we will refer to $U$ as the control and
denote by $(X_t^d, Y_t^d)$ the solution to the deterministic
system of equations \eqref{eq:deterministic_dyn}; that is, the system
\eqref{eq:sde_withcontrol} with $U(x,y)$ identically zero.  It is also
worth mentioning that for $U \in L^2([0,T],\RR)$, $t \mapsto
(X^U_t,Y^U_t)$ is continuous on $[0,T]$.

In analogy to the discussion at the start of the section, for $T>0$
and $z \in\RR^2$, we will
consider the  map $\Phi_{T,z}\colon L^2((0,T]; \RR) \rightarrow \RR^2$
defined by $U \mapsto (\Phi_{T,z}^{(1)},\Phi_{T,z}^{(2)})=
(X^U_T,Y^U_T)$ and $(X^U_0,Y^U_0)=z$.
Translating \cite{barlowNualart98} to our current setting, we obtain the
following theorem:
\begin{theorem}\label{thm:positiveCondition}
  $p_T(z,z')>0$ if and only if there exists a $U \in L^2([0,T],\RR)$
  so that $\Phi_{T,z}(U)=z'$ and furthermore
  the matrix
\begin{align}
M_{T,z}(U)=\begin{pmatrix}
\|D\Phi_{T,z}^{(1)}(U)\|^2_{L^2} & \langle D\Phi_{T,z}^{(1)}(U), D\Phi^{(2)}_{T,z}(U)\rangle_{L^2} \\
\langle D\Phi_{T,z}^{(1)}(U), D\Phi_{T,z}^{(2)}(U)\rangle_{L^2} & \|D\Phi_{T,z}^{(2)}(U)\|^2_{L^2}
\end{pmatrix}
\end{align}
is nondegenerate. Here $D$  represents the Frechet derivative.
\end{theorem}
The condition from
\cite{BenArousLeandre91,BenArousLeandre91B,barlowNualart98} that the
SDE under consideration have coefficients which are bounded with all
derivatives bounded can be removed by a standard localization
argument. The key step is knowing that $|(X_t,Y_t)|$ is almost surely
finite which follows from the Lyapunov function we constructed in the
preceding sections.

This matrix $M$ is just the product of the Jacobians introduced in the
motivating discussion at the start of this section, translated to
our current setting. To facilitate calculations and to better see this
connection, it is useful to
observe that for any $\eta \in \RR^2$
\begin{equation}\label{eq:jacobi_equiv_pos}
  \langle M_{T,z}(U) \eta, \eta \rangle=\int_0^T \langle J_{s,T}^U e_2,
  \eta \rangle^2 ds
\end{equation}
where   $e_2=(0,1)$ and $J_{s,t}^U$ is the Jacobi flow (the linearization of the
SDE/controlled ODE)
defined by
\begin{align*}
  J_{s,t}^U  =
  \begin{pmatrix}
    \frac{\partial \Phi^{(1)}_{t,z}(U)}{\partial \Phi^{(1)}_{s,z}(U)}
   &  \frac{\partial \Phi^{(1)}_{t,z}(U)}{\partial
     \Phi^{(2)}_{s,z}(U)}\\
& \\
    \frac{\partial \Phi^{(2)}_{t,z}(U)}{\partial \Phi^{(1)}_{s,z}(U)}
   &  \frac{\partial \Phi^{(2)}_{t,z}(U)}{\partial  \Phi^{(2)}_{s,z}(U)}
  \end{pmatrix}\,.
\end{align*}

We now state a simple condition which ensures the nondegeneracy of
$M$. It captures the fact that as long as there is some twist in
$J_{s,s+\epsilon}$ then $J_{s,s+\epsilon}e_2$ and $e_2$ will not be
co-linear, and hence $\ip{J_{s,s+\epsilon}e_2,\eta} + \ip{e_2,\eta}\neq 0$
for any $\eta\neq 0$. Combining this with the continuity of $s \mapsto
J_{s,T}$ produces the desired nondegeneracy of
\eqref{eq:jacobi_equiv_pos}. The following lemma follows this outline,
providing a condition which ensure such that the system has the
desired twist.
\begin{proposition}\label{prop:Mnondeg}
  To ensure the nondegeneracy of $M_{T,z}(U)$, it is sufficient that
  there exist a  $t_0 \in [0,T]$ so that
  $\Phi_{t_0,z}(U) \neq 0$.
\end{proposition}
\begin{proof}[Proof of Proposition \ref{prop:Mnondeg}]
Since $t \mapsto  \Phi_{t,z}(U)$  is continuous, we can without loss
of generality assume that $t_0 \in (0,T)$ and pick an $\epsilon>0$ so
$[t_0-\epsilon,t_0] \subset (0,T)$ and  $\Phi_{t,z}(U)>0$ for all $t
\in [t_0-\epsilon,t_0]$.

The nondegeneracy of $M_{T,z}(U)$ is equivalent to
\begin{align*}
\inf_{\eta \in \RR^2:|\eta|=1} \langle M_{T,z}(U) \eta ,  \eta \rangle > 0\,.
\end{align*}
In light of \eqref{eq:jacobi_equiv_pos}, we see that
\begin{align*}
    \langle M_{T,z}(U) \eta, \eta \rangle\geq\int^{t_0}_{t_0-\epsilon} \langle J_{t,T}^U e_2,
  \eta \rangle^2 dt =\int^{t_0}_{t_0-\epsilon}  \langle J_{t,t_0}^U e_2,
  (J_{t_0,T}^U)^*\eta \rangle^2 dt
  \end{align*}
where $(J_{t_0,T}^U)^*$ is the adjoint of the matrix
$J_{t_0,T}^U$. Combining these two observations, we see that for some positive constant $c$, depending on $U$,
\begin{align*}
\inf_{\eta \in \RR^2:|\eta|=1} \langle M_{T,z}(U) \eta ,  \eta
\rangle \geq c \inf_{\eta \in \RR^2:|\eta|=1}
\int_{t_0}^{t_0+\epsilon}  \langle J_{t,t_0}^U e_2, \eta \rangle^2 dt\,.
\end{align*}
Since $(t,\eta)\mapsto \langle   J_{t,t_0}^U e_2, \eta \rangle$ is
jointly continuous, it is sufficient to  show that
\begin{align}
  \label{eq:ipPos}
\langle J_{t,t_0}^U e_2, e_2^\perp \rangle \neq 0\quad \text{ for all $t \in [t_0-\epsilon,t_0)$}
\end{align}
for some positive $\epsilon$ where $e_2^\perp$ is the standard choice
of vector perpendicular to $e_2$. Since $J_{t_0,t_0}^U
e_2=e_2$, this guarantees that for every given $\eta \neq 0$, one has
$\langle   J_{t,t_0}^U e_2, \eta
\rangle\neq 0$ for $t$ in some open interval of $[t_0-\epsilon,t_0]$.
The continuity in  $\eta$ then ensures the infimum over all $\eta$
with $|\eta|=1$ is still positive.

To establish \eqref{eq:ipPos}, we appeal directly to the equations. We
see that as long as $\Phi_{t,z}(U)\neq 0$ to $t \in
[t_0-\epsilon,t_0]$ then $\langle J_{s,r} \eta,
\eta^{\perp}\rangle\neq 0$ for all $s,r$ with $t_0-\epsilon\leq
s<r\leq t_0$ and all $\eta\neq 0$. This is  due to the fact that
as long as the trajectory is not at zero, the linearization rotates
any vector a nontrivial amount over any interval of time. In
particular, the linearization satisfies the equation $\partial_t
J_{s,t} = A_t J_{s,t}$ for $t \geq s$ with $J_{s,s}$ equal to the
identity matrix and
\begin{equation}\label{eq:twist}
A_t=\begin{pmatrix}
2X_t & -2Y_t\\
2Y_t & 2X_t
\end{pmatrix}
=
2R_t\begin{pmatrix}
\cos \theta_t & -\sin \theta_t\\
\sin \theta_t & \cos \theta_t
\end{pmatrix}
\end{equation}
where $(X_t,Y_t)=(R_t \cos\theta_t,R_t \sin \theta_t)$.
\end{proof}

\begin{remark}The proof of Proposition~\ref{prop:Mnondeg} gives very appealing
  intuition concerning  the positivity of the transition
  density. Stochastic variation is injected at every moment of time in
  the $y$-direction.  However, the deterministic part of the
  flow is rotating at every point except the origin, as is seen from the
  calculation in \eqref{eq:twist}. This rotation ensures that there is
  stochastically independent variation in two linearly independent directions.
\end{remark}

As a consequence of Proposition~\ref{prop:Mnondeg}, to invoke
Theorem~\ref{thm:positiveCondition} to prove the positivity of
$p_t(z,z')$ for two given points $z,z' \in \RR^2$, we simply need to
find a control $U$ for which $\Phi_{t,z}(U)=z'$ and for which the path
does not spend all of its time at the origin. Since the path is
continuous in time, this last condition poses no restriction if either
$z$ or $z'$ is not the origin. If both $z$ and $z'$ are the origin, it
is still elementary to find a control satisfying the second condition
which still also satisfies $\Phi_{t,z}(U)=z'$.

We collect these last observations in the following lemma which
combines Proposition~\ref{prop:Mnondeg} and Theorem~\ref{thm:positiveCondition}.
\begin{corollary} The transition density $p_t$ satisfies $p_T(x,y)>0$
  for a given $T>0$ and $x,y \in \RR^2$ if  there exist a $U \in L^2([0,T],\RR)$ such that
  $\Phi_{T,x}(U)=y$ and there exists an $s \in [0,T]$ so that
  $\Phi_{s,x}(U)\neq 0$. Similarly, $p_T(x,y)=0$ if there exists no
  $U \in L^2([0,T],\RR)$ such that $\Phi_{T,x}(U)=y$.
\end{corollary}

We now build the required controls.
%
%
%
%
%
All the controls we design will take the form $U_t=u(X_t^U,Y_t^U,t)$ for
some piecewise smooth $u:\RR^2 \times[0,T] \rightarrow \RR$. At first
glance, this might seem an implicit specification which risks being ill-defined, since  $(X^U_t,Y_t^U)$ depends on the function $U_t$ we define
through \eqref{eq:sde_withcontrol}. However, a moment's reflection shows
this not to be the case, since in this setting  $(X^U_t,Y_t^U)$
can be defined in a self-contained way as the solution to
the ODE
\begin{align*}
\dot X_t^{U}&=(X_t^{U})^2-(Y_t^{U})^2 \notag \\
\dot Y_t^{U}&=2X^{U}_tY_t^{U} +u(X^{U}_t, Y_t^{U},t)\,.
\end{align*}
Then, with this solution in hand, one can define $U_t=u(X_t^U,Y_t^U,t)$.
\begin{figure}
 \centering
\begin{tikzpicture}[scale=.4]
\draw[color=white,thick] (5,10)--(5,-.5)--(-5,-.5)--(-5,10) --(5,10);
\draw[thick,color=gray] (-5,0) -- (5,0) node[right] {$x$};
\draw[thick,color=gray] (0,10) node[above] {$y$} --  (0,-.5);
\draw[-,color=red,thick,opacity=.75]
plot[id=c1,domain=1.5:2,smooth,samples=150] function {4*(x-1.5)*(x-1.5)};
\draw[-,color=red,thick,opacity=.75]
plot[id=c2,domain=-2.5:2.5,smooth,samples=150] function {sqrt(6.25-x*x)+2.5} ;
\draw[-,color=red,thick,opacity=.75]
plot[id=c3,domain=-2.5:-1,samples=150] function {-sqrt(6.25-x*x)+2.5};
\draw[-,color=red,thick,opacity=.75]
plot[id=c4,domain=2:2.5,samples=150] function {-sqrt(6.25-x*x)+2.5};
\draw[-,color=red,thick,opacity=.75]
plot[id=c5,domain=-5:-3,samples=150] function{-sqrt(25-x*x)+5};
\draw[-,color=red,thick,opacity=.75]
plot[id=c6,domain=-5:-2,samples=150] function{sqrt(25-x*x)+5};
\draw[-,color=red,thick,opacity=.75]
plot[id=c7,domain=-1:-.75,smooth,samples=150] function {-1.028*sqrt(-x-.75)+.75};
\draw[-,color=red,thick,opacity=.75]
plot[id=c8,domain=-2:-.75,smooth,samples=150] function {7.9*sqrt(-x-.75)+.75};
\draw[fill] (1.5,0) circle [radius=0.075];
\node [below] at (1.5,0) {$(x_0,y_0)$};
\draw[fill] (2,1) circle [radius=0.075];
\draw[fill] (-1, .236) circle [radius=0.075];
\draw[fill] (-2, 9.58) circle [radius=0.075];
\draw[fill] (-3,1) circle [radius=0.075];
\node [left] at (-3,.75) {$(x_*,y_*)$};
\end{tikzpicture}
\label{fig:control}
\caption{A representative example of a control path used to move from
any point in $\RR^2$ to any point in the left-half plane.}
\end{figure}
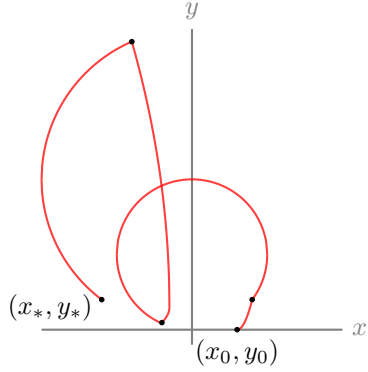
\begin{lemma}\label{lem:control_step}
  Let $z_*=(x_*, y_*)$ with $x_*<0$ be fixed.  There exists a finite
  $T_*(z_*)>0$ such that for all $z_0=(x_0,y_0) \in \RR^2$ and for all
  $T>T_*(z_*)$, there exists a control $U \in L^2([0,T], \mathbb{R})$
  for which the controlled system \eqref{eq:sde_withcontrol} satisfies
\begin{align*}
  (X_0^U, Y_0^U)=z_0, \quad (X^U_{T}, Y^U_{T})=z_*
\end{align*}
and such that $M_{T,z_0}(U)$ is nondegenerate.
\end{lemma}
\begin{proof}[Proof of Lemma \ref{lem:control_step}]
  We begin with the case of $z_*=(x_*,y_*)$ with $y_*\neq 0$. The
  remaining case will be considered at the end of the proof. The critical feature
  of $z_*$ with $y_*\neq 0$ is that there exists a deterministic orbit
  which begins on the $y$-axis and flows in finite time to the point $z_*$.

  Let $\mathcal{B}$ be the closed ball of radius $|x_*|/3$ about the origin
  intersected with the  negative half-plane $\{ (x,y) : x \leq 0\}$.
  We begin by setting $U_t=\sgn^+(y_0)$ for $t\in [0,1]$. (Here
  $\sgn^+(x)=x/|x|$ for $x \neq 0$ and 1 if $x=0$.)  Define $T_1$ to
  be the first time $t\geq 1$ such that $(X_t^U,Y_t^U) \in
  \mathcal{B}$ where $(X^U_0,Y^U_0) = z_0$.  We will set $U_t=0$ for
  $t \in [1,T_1+s]$ where $s\geq 0$ is a parameter we will vary in our
  construction. By driving with $|U_t|=1$ on the time interval
  $[0,1]$, we have ensured that $Y_1^U\neq 0$, which in turn implies
  that $T_1 \leq T_1^*$ for some finite, $z_0$-independent constant
  $T_1^*$. (In light of the discussion at the end of Section~\ref{sec:det_eq}, $T_1^*\leq 1+\frac{6}{|x_*|}$.) Let $\mathcal{C}_*$ be the orbit of the deterministic system which
  passes through the point $z_*$ but which is not contained in the set
  $\mathcal{B}$. Recall that this orbit is a circle
  tangent to the origin. We now define $T_2$ to be the infimum over
  time $t \geq T_1+s$ such that $(X_t^U,Y_t^U) \in \mathcal{C}_*$. If we define
  $U_t= \sgn^+(Y_t^U)M -2 X_t^U Y_t^U$ for $t \in [T_1+s,T_2]$, then
  $\dot Y_t^U = \sgn^+(Y_t^U)M$, and $ Y_t^U=Y_{T_1+s}^U
  +M\sgn^+(Y_{T_1+s}^U)\,(t-T_1-s)$ for $t \in
  [T_1+s,T_2]$. Hence for $M$ large enough, we  can ensure that $X_t$
  moves very little in the time it takes $Y_t$ to grow sufficiently to
  cross $\mathcal{C}_*$. This allows us to prove that $T_2 - (T_1+s)$ is bounded from
  above for any sufficiently large fixed choice of $M$ with a bound
  which is  independent of $s$ and $z_0$
  since $(X_{T_1+s}^U,Y_{T_1+s}^U) \in \mathcal{B}$. For the same
  reason, by choosing $M$ large enough we can ensure that $|X_{T_2}^U|
  \leq 2|x_*|/3$. Fixing such an $M$, we define $T_3$ to be the
  infimum of times $t> T_2$ such that $(X_t^U,Y_t^U) = z_*$.  For $t
  \in [T_2,T_3]$, we set $U_t=0$. Since $M$
  is fixed and $|X_{T_2}^U| \leq 2|x_*|/3$, clearly $T_3-T_2$ is
  bounded uniformly for all $s \geq 0$ and $z_0 \in \RR^2$. If we view
  $T_3-(T_1+s)$ as a function of $(X_{T_1+s}^U,Y_{T_1+s}^U)$, then it is
  continuous, since the governing ODEs depend continuously on their
  initial conditions. Since $(X_{T_1+s}^U,Y_{T_1+s}^U) \in
  \mathcal{B}$, which is a compact set, we know that there exists a
  finite bound $\tau^*$ so that $T_3-(T_1+s) \leq \tau^*$ for all
  $(X_{T_1+s}^U,Y_{T_1+s}^U) \in \mathcal{B}$.  Because
  $(X_{T_1+s}^U,Y_{T_1+s}^U)$ is a continuous function of $(z_0,s)$
  and $T_1$ an continuous function of $z_0$, we see that $(z_0,s)
  \mapsto T_3$ is a continuous function. Since $T_1$ and $T_3-T_2$
  are bounded uniformly in $(z_0,s)$, we conclude that if
  $T:(z_0,s)\mapsto T_3$, then there exists a $T_*$ so that $T(z_0,0)
  \leq T_*$ for all $z_0$. Since $s\mapsto T(z_0,s)$ is continuous and
  $T(z_0,s) \rightarrow \infty$ as $s \rightarrow \infty$, we conclude
  that for any $t \geq T_*$ and $z_0 \in \RR^2$ there exists a
  $s(t,z_0)$ so that $T(z_0,s(t,z_0))=t$. The control $U$ constructed
  corresponding to this choice of $s$ is the desired control. It is
  clearly in $L^2([0,t];\RR)$ since it is uniformly bounded.

  We now return to the case of $z_*=(x_*,y_*)$ where $y_*=0$. Let
  $(X_t^U,Y_t^U)$ be the solution with control $U_t=-2 X_t Y_t
  -1$ and initial condition $(X_0^U,Y_0^U)=z_*$. Since this is
  in fact  a flow, its solutions also exist backward in time. It is
  clear from the choice $U$ that $Y_{-1}^U =1\neq 0$. For future reference, define $z_1=(X_{-1}^U,Y_{-1}^U)$. Hence by the
  first part of the proof, there exists a $T_*$ so that for any
  $s\geq T_*$ there exists a control $U$ so that $(X_0^U,Y_0^U)=z_0$
  and $(X_s^U,Y_s^U)=z_1$. Thus if we define
  \begin{align*}
    \tilde U_t =\begin{cases}
      U_t & \text{if } t \in [0,s]\\
    -2 X_t Y_t  - 1 & \text{if } t \in (s,s+1]
    \end{cases}\,.
  \end{align*}
By the choice of $z_1$,  $(X_{T}^{\tilde U},Y_{T}^{\tilde U})=z_*$
as required if we set $T=s+1$.

Lastly we observe that none of the control paths employed are
identically zero for all time. Hence Lemma~\ref{prop:Mnondeg}, implies
that  $M_{T,z_0}(\tilde U)$ is non-degenerate.
\end{proof}

Combining the results from this section, we get the following
Proposition which is the analogue of
Proposition~\ref{prop:posDensityEliptic}.
\begin{proposition}\label{prop:posDensityHypoEliptic}
 If $\sigma_y>0$ but $\sigma_x=0$, then for any $z_*=(x_*,y_*) \in \RR^2$  with $x_*<0$ there exists a
$T_*$ so that for any $t>T_*$ one has $p_t(z,z_*) >0$ for any $z \in \RR^2$. Furthermore,
if $z_0=(x_0,y_0) \in \RR^2$  with $x_0<0$  and  $z_1=(x_1,y_1) \in
\RR^2$  with $x_1\geq 0$ than one has $p_t(z_0,z_1)=0$ for any $t>0$.
\end{proposition}
\begin{proof}[Proof of Proposition~\ref{prop:posDensityHypoEliptic}]
  As already outlined, the positivity claim follows from
  Theorem~\ref{thm:positiveCondition}, because
  Proposition~\ref{lem:control_step} guarantees the existence of a
  control with the needed properties.  The fact that $p_t(z_0,z_1)=0$
  when $x_1$ is strictly positive also follows from
  Theorem~\ref{thm:positiveCondition}, provided we can show that there is no
  control which moves one from $z_0$ to $z_1$. To see this, observe
  that except for the fixed point at the origin, the vector field for any
  control always points toward the left half-plane along the
  $y$-axis. Hence it is impossible to leave the left half-plane. The
  fact that $p_t(z_0,z_1)=0$ when $x_1=0$ follows from the strict
  positivity of $p_t(z_0,z_1)$ when $x_1<0$ and from the continuity of  $p_t(z_0,z_1)$.
\end{proof}

\subsection{Positivity of the Invariant Measure: Proof of Theorem~\ref{thm:aboutMu}
}
\label{sec:IM}

Assuming that $\sigma_y >0$, we know that $P_t(z,\cdot)$ is absolutely
continuous with respect to Lebesgue measure and has a smooth density. Hence if $\mu$ is an
invariant measure (and therefore we have $\mu=\mu P_t$ for any $t >0$), we see that
$\mu$ also has a smooth density, $m$, with respect to Lebesgue measure. The invariance implies that for all $z \in
\RR^2$ and $t>0$
\begin{equation}\label{densityInv}
  m(z)=\int_{\RR^2} p_t(z',z) m(z')dz'
\end{equation}
where $p_t$ is the density of $P_t$.

Let $z$ be a point such that for some $t>0$, $p_t(z',z)>0$ for all $z'
\in \RR^2$. Since $m$ integrates to one and is smooth, there must
exist some open set $A$ such that $m(z')
>0 $ for all $z' \in A$. Combining this observation with \eqref{densityInv}, we get
\begin{align*}
  m(z) \geq \int_A p_t(z',z)m(z') dz' >0\, .
\end{align*}
Hence we deduce that $m(z)$ is positive at any point $z$ which satisfies
the stated assumption. Applying this result to the information on the
positivity of $p_t$ in
Proposition~\ref{prop:posDensityEliptic} and
Proposition~\ref{prop:posDensityHypoEliptic}, we obtain the conclusions
about the positivity of $m(z)$ in Theorem~\ref{thm:aboutMu}.

To deduce the statements that $m(z)=0$ for $z=(x,y)$ with $x \geq 0$
if $\sigma_y>$ and $\sigma_x=0$, we use
Proposition~\ref{prop:posDensityHypoEliptic}, which states that if $w
\in H_+ = \{z=(x,y) \in \RR^2 : x \geq 0\}$, then $p_t(\zeta,w)=0$ for
all $t >0$ if $\zeta \in H_-$, where $H_-$ is defined to be the complement
of $H_+$. This implies that if $z \in H_+$, then
\begin{align*}
  m(z) = \int_{H_+} p_t(w,z) m(w) dw\,.
\end{align*}
Integrating this expression over $H_+$ and interchanging the order of
integration, we get
\begin{align*}
\int_{H_+}  m(z) dz  = \int_{H_+} \phi(w) m(w) dw\,.
\end{align*}
where $\phi(w)=\int_{H_+}  p_t(w,z) dz$. Since $m(z), \phi(z) \geq 0$
for all $z$, this implies that for Lebesgue-almost-every $z \in H_+$, either
$m(z)=0$ or $\phi(z)=1$. Yet from
Proposition~\ref{prop:posDensityHypoEliptic}, we know that given any
$w \in H_-$, there exists a time $t>0$ so that $p_t(z,w) >0$ for all $z
\in \RR^2$. Because $z\mapsto p_t(z,w)$ is continuous, we know
$\phi(w) <1$ for every $w \in H_+$. This implies that $m(z)=0$ for
almost every $z \in H_+$. Since $m(z)$ is continuous, this forces $m(z)=0$ for
all $z \in H_+$. This completes the proof of Theorem~\ref{thm:aboutMu}.

\section{Minorization and Geometric Ergodicity}
\label{sec:ergodicity}

We now establish the minorization
condition we need to complete the proof of Theorem~\ref{thm:main}. Specifically, we seek a probability measure $\nu$ and
positive constants $\alpha$, $R$, and $T$ so that
 \begin{align}\label{eq:minorization}
  \inf_{\{ z \in \RR^2 : |z|\leq R\}} P_T(z,\ccdot) \geq \alpha \nu(\ccdot)
\end{align}
and $R > K_T$ where $K_T$ is the constant from
Lemma~\ref{lemma:LyapimpliesA1}. This condition is a localized version
of the classical Doeblin condition and central to the theory of Harris
chains \cite{HarrisESMMP56,MeynTweedie1993,hairerMattingly08}. While
the Lyapunov condition ensures the existence of an invariant
measure and guarantees sufficiently rapid returns to the ``center of
phase space'' to produce geometric mixing, the minorization condition
ensures the existence of probabilistic mixing.

To summarize our current situation, we pause to prove the following
intermediate result.
\begin{theorem}\label{thm:mainIntermediate}
  If the minorization condition holds from \eqref{eq:minorization},
  then the Markov semigroup $P_t$ generated by \eqref{eq:sde}
  satisfies the conclusions of Theorem~\ref{thm:main}.
\end{theorem}

\begin{proof}[Proof of Theorem~\ref{thm:mainIntermediate}]
  By Theorem~1.3 in \cite{hairerMattingly08}, there exist constants $\bar{\alpha}\in(0,1)$ and $\beta>0$ such that
$\rho_\beta(\mu_1P_T,\mu_2P_T)\leq\bar{\alpha}\rho_\beta(\mu_1,\mu_2)$.
Results
 such as this are quite classical. Other proofs can be found, for example, in \cite{MeynTweedie1993}.
Combining this estimate with Proposition~\ref{prop:superLyapReg} immediately implies
that for any $n \in \{0\}\cup \NN$
\begin{align}
\rho_\beta(\mu_1P_{nT},\mu_2P_{nT})&\leq \bar{\alpha}^{n-1}\rho_\beta(\mu_1P_{T},\mu_2P_{T})\leq \bar{\alpha}^{n}\Bigl(\frac{1+\beta K_T}{\bar{\alpha}}\Bigr)\rho_0(\mu_1,\mu_2)\,.\label{eq:mainDiscreteIneq}
\end{align}
To extend this estimate  to an arbitrary $t\geq 0$, we define a
nonnegative integer $n$ and $\tau \in (0,1)$ so that $t=nT+\tau$ and
observe that
\begin{align*}
  \rho_\beta(\mu_1P_{t},\mu_2P_{t})&=\rho_\beta(\mu_1P_{\tau}P_{nT},\mu_2P_{\tau}P_{nT}) \leq \bar{\alpha}^{n}\Bigl(\frac{1+\beta K_T}{\bar{\alpha}}\Bigr)\rho_0(\mu_1P_\tau,\mu_2P_\tau)\\
  &\leq \bar{\alpha}^{n}\Bigl(\frac{1+\beta
    K_T}{\bar{\alpha}}\Bigr)\rho_0(\mu_1,\mu_2)\leq
  \bar{\alpha}^{\frac{t}{T}}\Bigl(\frac{1+\beta
    K_T}{\bar{\alpha}^2}\Bigr)\rho_0(\mu_1,\mu_2)
\end{align*}
As noted in Remark~\ref{rem:easyDirectionInequality}, for
any $\beta' \geq0$ there exists a constant $C$ so that
$\rho_{\beta'}(\nu_1,\nu_2)\leq C
\rho_\beta(\nu_1,\nu_2)$ for all probability
measure $\nu_i$. This completes the proof.
\end{proof}
%

\subsection{Minorization  when $\sigma_x>0$}
\label{sec:elliptic2}
Now since for each $t >0$, $(z,z') \mapsto p_t(z,z')$ is continuous
and everywhere positive, it is elementary that for any $R >0$ there exists a positive
constant $\alpha=\alpha(R,t) >0$ so that $\inf \{ p_t(z,z') :
z,z' \in \RR^d, |z|, |z'| \leq R\} \geq \alpha$.
The minorization condition follows
immediately from this, since for any measurable set $A$
\begin{align*}
  P_t(z,A) = \int_A p_t(z,z')dz' \geq \alpha \textrm{Leb}(A\cap
  B_R(0)) \eqdef \alpha{ \textrm{Leb}(B_R(0))} \nu(A)
\end{align*}
where $\textrm{Leb}$ is Lebesgue measure and $\nu(A)=\textrm{Leb}(A
\cap B_R(0))/ \textrm{Leb}(B_R(0))$.

\subsection{Minorization when $\sigma_x=0$}
\label{sec:minor_pos_noise-y}

We now state and prove a lemma which shows that the needed minorization
condition follows quickly from continuity and a relaxed positivity
assumption. In Section~\ref{sec:hypoelliptic}, this relaxed positivity assumption was shown to hold by using a very explicit control theory argument coupled with some stochastic analysis.

\begin{lemma}\label{lemma:minorization}
Let $P(z,dz')$ be a Markov transition kernel on $\RR^d$ such that
$P(z,dz')=p(z,z')dz'$ with $p\colon \RR^d \times \RR^d
\rightarrow [0,\infty)$ jointly continuous. If  there exists
$z^*\in\mathbb{R}^d$ and  a compact set $B$ such that  for all $z\in B$, $p(z,z^*)>0$
there exist $\alpha\in(0,1)$ and a probability measure $\nu$ such that
\begin{equation*}
\inf_{z\in B} P(z,\ccdot)\geq\alpha \ \nu(\cdot)\,.
\end{equation*}
\end{lemma}

\begin{proof}
Since $z \mapsto p(z,z^*)$ is continuous, it achieves its minimum on the
compact set $B$. Since  $p(z,z^*)$ is strictly
positive for all $z \in B$, we know that for some $\alpha \in (1/2,0)$,
$p(z,z^*) >2 \alpha$ for all $z \in B$. By the joint continuity of
$p$, for each $z\in B$ there exists a $\delta_z>0$ be such that
$p(y,x) > \alpha$ for all $(y,x) \in B_{\delta_z}(z) \times
B_{\delta_z}(z^*)$. Since $\{B_{\delta_z}(z) \colon z \in B\}$ is an open
cover of the compact set $B$, we can extract a finite subcover. Let $K$ be the collection of points $z$ associated with this finite
subcover. If $\delta = \min\{ \delta_z : z \in K\}$, then $\delta>0$
since $\delta_z>0$ and $K$ is finite. Now since
\begin{align*}
  B \subset \bigcup_ {z \in K} B_{\delta_z}(z)
\end{align*}
we have that $p(y,x) > \alpha$ for all $(y,x) \in B \times
B_\delta(z^*)$. Then $P(y,A) \geq \alpha \textrm{Leb}( B_{\delta(z)}(z))  \nu(A)$, where $\nu(A) =
\textrm{Leb}(A \cap B_\delta(z^*))/\textrm{Leb}( B_{\delta(z)}(z))$ and
$\textrm{Leb}$ is Lebesgue measure on $\RR^d$, since
\begin{align*}
  P(y,A) =\int_A p(y,x)dx \geq  \int_{A\cap B_\delta(z^*)} p(y,x)dx
  \geq  \alpha  \textrm{Leb}(A \cap B_\delta(z^*))\,.
\end{align*}
\end{proof}


\section{Conclusion}
\label{sec:conclusion}

We describe a general methodology for building a Lyapunov function in
a setting where the global stability of the systems requires flux of
probability into regions which are clearly dissipative from the rest of
phase space. We are most interested in problems, like the example
considered here, where the noise plays an essential role in creating
this transport is some regions.  The algorithm makes use of local
Lyapunov functions, which are constructed as solutions to Poisson
equations in different regions, and are then patched together to form
one global Lyapunov function.  We apply these techniques to one
specific example in the plane to illustrate how the addition of noise
gives rise to an invariant probability measure for a system whose
purely deterministic dynamics exhibit instability.  Furthermore, our
resulting ``super'' Lyapunov function enables us to extract a stronger
convergence than what is usually proved in the Harris-chain
setting---indeed, an exponential convergence independent of initial
condition---to this equilibrium measure. En route to proving this
convergence, we employ explicit control-theoretic constructions and we
rely on the tools of Malliavin calculus.  It is our hope that the
simple and specific applications of control theory and Malliavin
calculus in our model problem will be of independent interest.

Of course, further work remains to be done: the application of these
methods to other examples, for instance, and the development of more
general theorems about noise-induced stabilization. In particular, in
the planar system we consider, patching the local Lyapunov functions
turns out to be one of the most delicate and important parts of the
proof. Therefore, it would be especially interesting to find more
general approaches to the problem of patching Lyapunov functions, and
more general conditions under which it can be done successfully.

When our construction works, it is likely to produce a Lyapunov
function which provides strong control over the excursions towards
infinity and a nearly sharp rate for the convergence to
equilibrium. However, the construction of such a function is
laborious. It would be interesting to obtain a simpler ``partial fluid
limit'' which captures only the minimal stochasticity at infinity
needed to stabilize the system. This may arise as an extension to our
work in the direction of
\cite{DupuisWilliams1994,HuangKontoyiannisMeyn2001,FortMeynMoulinesPriouret2008,Meyn2008}. Such
an approach might allow simpler proofs of stabilization without
necessarily proving the existence of strong Lyapunov function.

\appendix

\section{Comparison Proposition}
\label{sec:compr-prop}

\begin{proposition}\label{prop:comparison}
  Suppose $f\in C(\RR)$ is a non-increasing function and that
  $\phi(t)$ and $\psi(t)$ are $C^1$ functions on $\RR$ satisfying
 $\phi(0)=\psi(0)$ and  $\phi'(t)\leq f(\phi(t))$,
 $\psi'(t)=f(\psi(t))$  for all $t \geq 0$
 then $\phi(t)\leq\psi(t)$ for all $t\geq0$.
\end{proposition}
\begin{proof}[Proof of Proposition~\ref{prop:comparison}]
  For all $0\leq r\leq t$, we have that
  \begin{align*}
    \phi(t)\leq \phi(r) +\int_r^t f(\phi(s))ds\quad\text{and}\quad
    \psi(t) =\psi(r) +\int_r^t f(\psi(s))ds
\end{align*}
which implies
\begin{equation*}
  \psi(t)-\phi(t) \geq(\psi(r)-\phi(r)) +\int_r^t (f(\psi(s))-f(\phi(s)))ds\,.
\end{equation*}
Let $T_1=\inf\{t>0: \psi(t)-\phi(t)<0\}$.  Suppose for contradiction
that $T_1<\infty$.  Then by continuity, $\psi(T_1)-\phi(T_1)=0$ and
there exists $T_2\in(T_1,\infty)$ such that for all $t\in(T_1,T_2)$,
$\psi(t)-\phi(t)<0$.  Then for all $t\in(T_1,T_2)$,
\begin{equation}
  \label{eq:psiphi}
  \psi(t)-\phi(t) \geq(\psi(T_1)-\phi(T_1)) +\int_{T_1}^t (f(\psi(s))-f(\phi(s)))ds\,.
\end{equation}
Now since $f$ is non-increasing, $\psi(t)<\phi(t)$ implies that
$f(\psi(t))\geq f(\phi(t)$.  Hence this combined with
\eqref{eq:psiphi} implies that for all $t\in(T_1,T_2)$,
$\psi(t)-\phi(t)\geq0$.  This is a contradiction.  Hence $T_1$ must be
infinite and $\phi(t)\leq\psi(t)$ for all $t\geq0$.
\end{proof}

\bibliographystyle{alpha}

\bibliography{./refs}

\end{document}